\documentclass[11pt]{amsart}
\usepackage{amsmath,amssymb,amsthm,amsfonts,color,longtable}
\usepackage{hyperref}
\usepackage{latexsym}
\usepackage{array}
\usepackage{verbatim}

\numberwithin{equation}{section}

\theoremstyle{plain}

\newtheorem*{acknowledgements}{Acknowledgements}

\numberwithin{equation}{section}
\newtheorem{thm}{Theorem}[section]
\newtheorem{prop}[thm]{Proposition}
\newtheorem{lemma}[thm]{Lemma}
\newtheorem{cor}[thm]{Corollary}

\theoremstyle{remark}
\newtheorem{remark}[thm]{Remark}

\newtheorem{example}[thm]{Example}

\theoremstyle{definition}
\newtheorem*{notation}{Notation}

\providecommand{\abs}[1]{\left\lvert#1\right\rvert}

\providecommand{\ip}[1]{\langle#1\rangle}
\providecommand{\iFT}[1]{\mathcal{F}^{-1}\left(#1\right)}
\providecommand{\ccdot}[0]{\!\cdot\!}

\renewcommand{\div}{{\rm div}}
 
\newcommand{\Z}{\mathbb{Z}}  
  
\newcommand{\N}{\mathbb{N}}
\newcommand{\R}{\mathbb{R}}
\newcommand{\supp}{\rm supp}

\title{Long Time Stability for Solutions of a \texorpdfstring{$\beta$}{beta}-plane Equation}
\author{\vspace{-.1cm}Tarek M. Elgindi, Klaus Widmayer}
\subjclass[2010]{76B03, 76B15, 76U05, 35B34, 35Q31, 35Q35}
\date{\today}

\address{Department of Mathematics, Princeton University, Princeton, NJ 08544, USA}
\email{tme2@math.princeton.edu}

\address{Courant Institute of Mathematical Sciences, 251 Mercer Street, New York 10012 NY, USA}
\email{klaus@cims.nyu.edu}

\begin{document}
\vspace{-3cm}
\begin{abstract}\vspace{-.2cm}
We prove stability for arbitrarily long times of the zero solution for the so-called $\beta$-plane equation, which describes the motion of a two-dimensional inviscid, ideal fluid under the influence of the Coriolis effect. The Coriolis force introduces a linear dispersive operator into the 2d incompressible Euler equations, thus making this problem amenable to an analysis from the point of view of nonlinear dispersive equations. The dispersive operator, $L_1:=\frac{\partial_1}{\abs{\nabla}^2}$, exhibits good decay, but has numerous unfavorable properties, chief among which are its anisotropy and its behavior at small frequencies. 
\end{abstract}

\maketitle \vspace{-0.9cm}

\tableofcontents

\section{Introduction}

While the study of partial differential equations arising in fluid mechanics has been an active field in the past century, many important and physically relevant questions remain wide open from the point of view of mathematical analysis. Even very fundamental questions such as global-in-time well-posedness remain unresolved. This is particularly true in the study of the long-time behavior of systems without dissipation, such as the Euler equations for incompressible flow,
\begin{align}
 \partial_t u +u\cdot\nabla u+\nabla p&=0 \quad \text{in} \; \R^n,\label{eq:Euler1}\\
 \div(u)&=0 \quad \text{in} \; \R^n,\label{eq:Euler2}\\
 \abs{u(x)}&\rightarrow 0 \;\;\;\text{as}\; \abs{x}\rightarrow \infty.\label{eq:Euler3}
\end{align}

These model the motion of a so-called ideal, incompressible fluid occupying the whole space $\mathbb{R}^n$, where $u$ is the (vector) velocity field of the fluid and $p$ is the (scalar) internal pressure.

The most physically relevant cases to study \eqref{eq:Euler1}-\eqref{eq:Euler3} are dimensions $n=2$ and $n=3$. While it isn't apparent from the system in this form, the mathematical analysis of \eqref{eq:Euler1}-\eqref{eq:Euler3} in two dimensions is very different from the analysis in three dimensions. This is easily seen when we pass to the vorticity formulation: Define $\omega:= \nabla \times u$, the curl of $u$ (in two dimensions $\omega=\partial_x u_2-\partial_{y} u_1$). It turns out that \emph{in two dimensions} $\omega$ satisfies a simple transport equation,
$$\partial_t\omega+u\cdot\nabla\omega=0.$$
On the other hand, the so-called Biot-Savart law allows us to recover $u$ from the (scalar) vorticity $\omega$ using the fact that $\div(u)=0$ as $u=\nabla^\perp (-\Delta)^{-1}\omega$.

The fact that $\omega$ is transported by a velocity field $u$ which is one degree smoother than $\omega$ (by the Biot-Savart law) allows one to prove that in two dimensions \eqref{eq:Euler1}-\eqref{eq:Euler3} is globally well-posed\footnote{This is not known in three dimensions, where $\omega$ is not simply transported.} in Sobolev spaces, $H^s$ for $s>2$. However, a priori the norm of the solutions may grow double exponentially in time. 

\subsection{The \texorpdfstring{$\beta$}{beta}-plane model}
One mechanism which helps to further stabilize the motion of an ideal fluid is the so-called \emph{Coriolis effect}, which arises when the fluid is described in a rotating frame of reference, in which it is seen to experience an additional force (called the Coriolis force). On a rotating sphere the magnitude of this effect varies with the latitude and is described by the Coriolis parameter $f$. A common approximation in oceanography and geophysical fluid dynamics (see \cite{GFD}, for example) is to assume that this dependency is linear. In the appropriate coordinate system the Coriolis parameter is then given as $f=f_0+\beta y$, where the variable $y$ corresponds to the latitude and $\beta$ is a parameter (from which the model gets its name). In two space dimensions, this is the setting of the current paper (see also \cite{AIP_1.3141499}). 

More explicitly, we are interested in studying the long-time behavior of solutions to the so-called \emph{$\beta$-plane equation}
\begin{equation}\label{eq:b-plane}
 \partial_t \omega +u\ccdot\nabla\omega=\beta\frac{R_1}{\abs{\nabla}}\omega,
\end{equation}
where $\omega:\R\times\R^2\rightarrow\R$, $u=\nabla^\perp(-\Delta)^{-1}\omega,$ $R_1$ denotes the Riesz transform in the first space coordinate, and $\beta$ is a real number determined by the strength of the Coriolis force. For simplicity, we shall take $\beta=1$. Henceforth we shall abbreviate the operator on the right hand side by 
$$L_1\omega:=\frac{R_1}{\abs{\nabla}}\omega.$$

As described above, this system models the flow of a two-dimensional inviscid fluid under the influence of the Coriolis effect -- we have written it here in the vorticity formulation. The asymmetry between the horizontal and vertical motion can be seen clearly in the linear term $L_1\omega$ on the right hand side of the equation. In more technical terms, this breaks the isotropy of the original Euler equation and is a core feature and difficulty of the model.

From a purely hyperbolic systems point of view, i.e.\ using only $L^2$ based energy estimates, one can show that solutions remain small in $H^k$ energy norms for a time span whose length is inversely proportional to the size of the initial data. There the operator $L_1$ does not play any role since $\ip{L_1\omega,\omega}_{L^2}=0$  and since $L_1$ commutes with derivatives. However, one can show that such energy norms are controlled by the $L^\infty$ norm. Since the linear part of the equation exhibits dispersive decay in $L^\infty$, this controllability of higher order energy norms by the $L^\infty$ norm of $\omega$ hints at a first possible improvement through dispersive decay. Indeed, as we already observed in our previous work on the related SQG equation \cite[Remark 3.3]{SQGpaper}, the smallness of solutions can be guaranteed for a longer\footnote{More precisely, for data of size $\epsilon$ smallness is guaranteed to hold up to a time of order $\abs{\log(\epsilon)}^{-1}\epsilon^{-2}$.} period of time. 

\subsubsection*{The Dispersive Approach}
In the present article we will improve this result and demonstrate smallness and decay of solutions to the $\beta$-plane equation for arbitrarily long times. More precisely, we will show (see Theorem \ref{thm:main}) that for any $M>0$ there exists an $\epsilon>0$ such that if the initial data are smaller than $\epsilon$ (in suitable spaces), then the corresponding solution to the $\beta$-plane equation \emph{stays small and decays} until at least time $t\sim\epsilon^{-M}$.

We approach the question of well-posedness of the $\beta$-plane system for small initial data from a perturbative point of view, i.e.\ we view \eqref{eq:b-plane} as a nonlinear perturbation of the linear equation 
\begin{equation}\label{eq:lin_eq}
 \partial_t \omega =L_1\omega.
\end{equation}
Due to the properties of the operator $L_1$ this puts us in the realm of nonlinear dispersive differential equations: solutions to \eqref{eq:lin_eq} decay in $L^\infty$ at a rate of $t^{-1}$. This rate of decay is not integrable, so obtaining good control of solutions to the nonlinear equation \eqref{eq:b-plane} for long periods of time is a challenge. However, it is the most decay one may hope for in a 2d model, so despite being anisotropic this model is not further degenerate in that sense (unlike the dispersive SQG equation, which has a lower rate of decay of $t^{-\frac{1}{2}}$ -- see \cite{SQGpaper}). 

However, for small frequencies $L_1$ is quite singular: a bound on third order derivatives in an $L^1$ based Besov space is required for the decay given by the linear propagator. For the full, nonlinear $\beta$-plane equation one is then led to the study of weighted norms with derivatives in order to guarantee the decay.

Further difficulties arise due to the lack of symmetries of this equation. In general, symmetries are a convenient way to obtain weighted estimates by commuting the associated generating vector fields with the equation. As noted above though, \eqref{eq:b-plane} does not exhibit a rotational symmetry. Only a scaling symmetry holds, which merely gives control\footnote{For more details we refer the reader to Remark \ref{rem:symm} below.} over a radial weight in $L^2$. This, however, is not enough to control the decay of solutions. Given the form of $L_1$, proving weighted estimates on the equation is therefore the key difficulty tackled in this article.

This way we are lead to a detailed analysis of the resonance structure of the equation (in the spirit of the works of Germain, Masmoudi and Shatah -- see for example \cite{MR2914945}). Here it is crucial to note that the nonlinearity $u\cdot\nabla\omega$ exhibits a null structure that cancels waves of parallel frequencies. However, the anisotropy of $L_1$ makes a precise study of the geometric structure of the resonant sets difficult or even infeasible.

It is also worth noting that, compared to many other nonlinear dispersive equations, this problem is highly quasi-linear in the sense that the nonlinearity has two orders of derivatives more than the linear operator $L_1$.

\subsubsection*{Related Works on Rotating Fluids}
The $\beta$-plane approximation has been used in numerous investigations into various types of waves in rotating fluids in two or three spatial dimensions. For the most part the theoretical efforts have been focused on \emph{viscous} scenarios, such as the rotating Navier-Stokes equations. The dispersion relation there is anisotropic, but less singular\footnote{In three dimensions it is given by $\frac{\xi_3}{\abs{\xi}}$. However, notice that even in 3d this only gives a dispersive decay rate of $t^{-1}$ -- see e.g.\ \cite{KohNS}.} than the one we study. The viscosity introduces a strong regularizing effect and makes the equations semilinear. In particular, Strichartz estimates can then be used to construct global solutions once the dispersion is sufficiently strong, and  to study limiting systems as the dispersion tends to infinity. For the rotating Navier-Stokes equations this was done by Chemin et al.\ (\cite{CDGG1999}, \cite{CDGG2002}, \cite{CDGG2006}) and Koh et al.\ (\cite{KohNS}).

Another important achievement in the study of both the rotating Euler equations and Navier-Stokes equations (in 3d) is the series of works of Babin-Mahalov-Nicolaenco (\cite{BMN1996}, \cite{BMN1997}, \cite{BMN1999}, \cite{BMN2000}), which focused on studying the case of fast rotation. They proved that as the speed of rotation increases (depending upon the initial data), the solution of the rotating Navier-Stokes equations becomes globally well-posed and the solution of the Euler equations exists for a longer and longer time. This was shown in many different settings, including various spatial domains such as three dimensional tori (\cite{BMN2000}) with non-resonant (\cite{BMN1996}, \cite{BMN1997}) and resonant dimensions (\cite{BMN1999}). Apart from the more singular dispersion relation, another key difference between the works of Babin-Mahalov-Nicolaenco and the present article is that we are able to prove almost global existence in the setting of the whole space \emph{without} viscosity, while their method does not give such a long time of existence and, more importantly, global existence without viscosity seems out of reach using those types of methods. In contrast, our argument does not apply to the periodic case.

\subsubsection*{Related Works in Dispersive Equations}
In recent years there has been a surge of interest in the question of dispersive effects in fluid equations. However, just as is the case for the classical dispersive equations such as the (Fractional) Nonlinear Schr\"odinger Equation, in many examples (e.g.\ \cite{CapWW}, \cite{G2dWW}) the dispersion relation is isotropic. This greatly helps with the resonance analysis and allows for an explicit understanding. In addition, rotational symmetry and scaling symmetries provide conservation laws for the solutions in weighted spaces.

An example of an anisotropic equation that bears closer similarity with the $\beta$-plane model is the KP-I equation, recently studied by Harrop-Griffiths, Ifrim and Tataru \cite{KP-I}. At the linear level, this 2d model produces the same decay rate as the $\beta$-plane equation and does not allow for a direct resonance analysis due to the anisotropy. However, in contrast to the $\beta$-plane equation, the KP-I equation does have enough symmetries to provide control of weighted norms and the main difficulty is to obtain a global decaying solution. Using the method of testing by wave packets, the authors were able to obtain global solutions for small, localized initial data.

\subsection{Plan of the Article}\label{sec:plan}
We start with the presentation of our main result, Theorem \ref{thm:main} in Section \ref{sec:main}. This is followed by a short discussion that places it in the relevant context. The rest of the article is then devoted to its proof.

We begin (Section \ref{sec:bootstrap}) by explaining the general bootstrapping scheme used to prove Theorem \ref{thm:main}. It proceeds through a combination of bounds on the energy, the dispersion and control of some weighted norms. The individual estimates employed here are the subject of the remaining sections.

Firstly, an energy inequality is established in Section \ref{sec:en_ineq}. This is a classical blow-up criterion, which shows that the growth of the energy norms is bounded by the exponential of the time integral of the $L^\infty$ norms of the velocity and vorticity.

We go on to discuss the dispersive decay given by the linear semigroup of the equation in Section \ref{sec:disp_est}. We explain the general dispersive result (with a proof in the appendix) and illustrate how we will make use of a more particular version, tailored to our bootstrapping argument in Section \ref{sec:bootstrap}.

The core of the article then is the proof of the weighted estimates in Sections \ref{sec:weights}, \ref{sec:weights_red} and \ref{sec:res}. We commence by explaining in Section \ref{sec:weights} the basics of our approach and give some insight into the difficulties we encounter. This is followed by an explanation of our strategy to prove the main weighted estimates (Propositions \ref{prop:added_derivative}, \ref{prop:highfreq} and \ref{prop:lowfreq}). 

In Section \ref{sec:weights_red} we embark on the proof with a few observations, which allow us to reduce to one main type of term, which is treated in Section \ref{sec:res}. There we illustrate first the various techniques that come in handy and then go on to apply them in the later parts of this section. In particular, the main part of Section \ref{sec:res} is a detailed resonance analysis and relies on the observations and ideas laid out beforehand.

In the appendix we gather the proofs of the full dispersive decay, of a corollary to our main result and of the global well-posedness for large data, as well as some computations that turn out to be useful.

\subsection{Notation}
For $j\in\Z$ we let $P_j$ be the Littlewood-Paley projectors associated to a smooth, radial bump function $\varphi: \R^2\to\R$ with support in the shell $\frac{1}{2}\leq \abs{\xi}\leq 2$ and its rescalings $\varphi_j(\xi):=\varphi(2^{-j}\xi)$, so that one has $\widehat{P_j g}(\xi)=\varphi(2^{-j}\xi)\hat{g}(\xi)$. Derived from these by summation are the operators $P_{\leq N}$, $P_{>N}$ etc. For example, $P_{>N}:=\sum_{2^j>N}P_j$ gives the projection onto frequencies larger than $N>0$ in Fourier space (and can alternatively be written using a smooth Fourier multiplier $\varphi_{>N}(\xi)$ with value 1 on frequencies larger than $N$ and supported on frequencies larger than $N/2$).

We use standard notation for derivatives, but make use of the slight abbreviation to denote by $\partial_\xi$ a scalar derivative in $\xi$ in any direction. Moreover, we let $D:=\sqrt{-\Delta}$.

For a set $A$ we denote by $\chi_A$ its characteristic function.

Whenever a parameter is carried through inequalities explicitly, we assume that constants implicit in the corresponding $\lesssim$ are independent of it. 

\section{Main Result}\label{sec:main}
The main result of this article is the following
\begin{thm}\label{thm:main}
 For any $M>0$ there exist $\epsilon_0>0$ and $k_0>0$ such that for all $\epsilon\leq\epsilon_0$ and $k\geq k_0$, if $\omega_0$ satisfies
\begin{equation*}
 \abs{D^2 x\omega_0}_{L^2},\abs{D^3 x\omega_0}_{L^2},\abs{\omega_0}_{H^k}\leq\epsilon,
\end{equation*}
then there exist $T\gtrsim\epsilon^{-M}$ and a unique solution $\omega(t)\in C([0,T];H^k)$ to the initial value problem
\begin{equation}\label{eq:IVP}
   \begin{cases}
   &\partial_t\omega+u\cdot\nabla\omega=L_1\omega,\\
   &u=\nabla^\perp(-\Delta)^{-1}\omega,\\
   &\omega(0)=\omega_0.
  \end{cases}
\end{equation}

Moreover, for $t\in [0,T]$ the solution decays in the sense that
\begin{equation}\label{eq:thmdecay}
 \abs{\omega(t)}_{L^\infty}+\abs{u(t)}_{L^\infty}+\abs{Du(t)}_{L^\infty} \lesssim \frac{\epsilon^{\frac{1}{4}}}{\ip{t}}
\end{equation}
and the energy as well as a weight with two or three derivatives on the profile $f(t)=e^{-L_1t}\omega(t)$ remain bounded:
\begin{equation}\label{eq:thmbounds}
 \abs{\omega(t)}_{H^k},\abs{D^2 xf(t)}_{L^2},\abs{D^3 xf(t)}_{L^2}\lesssim \epsilon^{\frac{1}{2}}.
\end{equation}
\end{thm}

\begin{remark}[Decay of Riesz transforms of $u$ and $\omega$]
 Note that $Du$ and $\omega$ only differ by a Riesz transform - since we obtain decay of $u$ and $\omega$ through $L^2$ based bounds on a weight on the profile, the proof of the theorem actually shows that any Riesz transforms of $u$ and $\omega$ decay at the same rate, respectively -- see also Remark \ref{rem:u_Du_decay}. The only reason we include both $Du$ and $\omega$ here is that they come up explicitly in the estimates.
\end{remark}

The question of stability of the zero solution was touched upon briefly in our previous work \cite{SQGpaper}, where we showed that smallness of $\omega$ in $H^k$ for $k>5$ is propagated for at least a time $T\sim\epsilon^{-2}\abs{\log(\epsilon)}^{-1}$. However, this result does not make use of any finer structure of the equation than the blow-up criterion for the energy and does not give decay of the $L^\infty$ norm.

We point out that even for large data global solutions to the $\beta$-plane equation \eqref{eq:b-plane} can be constructed for $\omega$ in $H^k$, $k>1$ -- see Theorem \ref{thm:GWP}. However, a priori, the $H^k$ norms of the solution could grow faster than double exponentially in time, which is completely different from the situation we have in the small data case. The proof of this result proceeds in analogy with that for the 2d Euler equation. We include it in Appendix \ref{sec:GWP} for ease of reference.

\begin{proof}[About the Proof of Theorem \ref{thm:main}]
 The local well-posedness of \eqref{eq:IVP} for $k>1$ is well-known. It can be seen by invoking the Sobolev embedding $H^k\hookrightarrow L^\infty$ for the exponential in the energy estimate \eqref{eq:energy_ineq}, for example. Thus the crucial part of the theorem is to obtain decay and existence on a ``long'' time interval $[0,T]$.
 
 For this we employ a \emph{bootstrapping argument}, in which a decay assumption (similar to, but weaker than \eqref{eq:thmdecay}) for some time interval gives the bounds \eqref{eq:thmbounds} on that time interval. This in turn will yield the stronger decay \eqref{eq:thmdecay}, which allows us to prolong the solution and then repeat the argument by the continuity of $\omega$, as long as $t\leq T$.
 
 We make use of Fourier space methods to understand the detailed structure of the nonlinearity through a resonance analysis, the null structure being a crucial element. The present rate of dispersive decay, $t^{-1}$, is not time integrable, which prevents us from proving estimates on a global time scale. More precisely, for the present quadratic nonlinearity, the bilinear estimates of type $L^2\times L^\infty\rightarrow L^2$ have a logarithmic growth when integrated in time. Together with a loss of derivatives in $L^\infty$ in our weighted estimates, this forces us to use a lot of derivatives for the energy estimates. However, the constants involved therein grow exponentially with the amount of derivatives, so that we are not able to close our estimates on an exponentially long time interval.

 As outlined above in \ref{sec:plan}, the proof is spread out over the rest of this article: We begin by demonstrating how a bootstrapping scheme combines our estimates for the energy, dispersion and weights to yield Theorem \ref{thm:main}. Subsequently, these ingredients are proved: Section \ref{sec:en_ineq} establishes energy estimates, after which we demonstrate in Section \ref{sec:disp_est} how to obtain and use the dispersive estimate for the linear semigroup $e^{L_1t}$. Finally, Sections \ref{sec:weights}, \ref{sec:weights_red} and \ref{sec:res} furnish the proof of the weighted estimates.
\end{proof}

\begin{remark}[Symmetries]\label{rem:symm}
 As mentioned in the introduction, apart from time translation, the only continuous symmetry the $\beta$-plane equation in \eqref{eq:IVP} exhibits is a scaling symmetry. The former gives the conservation of $\omega$ and $u$ in $L^2$,
$$\abs{\omega(t)}_{L^2}=\abs{\omega(0)}_{L^2},\quad \abs{u(t)}_{L^2}=\abs{u(0)}_{L^2}.$$
 As for the latter, one can see that if $\omega(t,x)$ solves \eqref{eq:b-plane}, then so does $\omega_\lambda(t,x):=\lambda^{-1} \omega(\lambda^{-1} t, \lambda x)$ for any $\lambda\neq 0$. Differentiating with respect to $\lambda$ at $\lambda=1$ gives the generating vector field $S$ of this symmetry as $S:=-1-t\partial_t+x\ccdot\nabla$, and we have $$\partial_t S\omega+u\ccdot\nabla S\omega+Su\ccdot\nabla\omega=L_1(S\omega).$$
 From this we obtain the $L^2$ conservation of $S\omega$: $\abs{S\omega(t)}_{L^2}=\abs{S\omega(0)}_{L^2}$. Moreover, through the commutation of $S$ with the linear semigroup $e^{L_1t}$ and the equation one can so obtain a bound on $\abs{x\ccdot\nabla f(t)}_{L^2}$ (under the appropriate decay assumptions).
\end{remark}

As an aside we note that from the proof of Theorem \ref{thm:main} it also follows that the Fourier transform of the profile also stays bounded in $L^\infty$ (with two derivatives):
\begin{cor}\label{cor:Linfty}
 In the setting of Theorem \ref{thm:main}, if one also has $\abs{\abs{\xi}^2\hat{\omega}_0(\xi)}_{L^\infty}\leq\epsilon$, then
  \begin{equation*}
   \abs{\abs{\xi}^2\hat{f}(t,\xi)}_{L^\infty}\lesssim\epsilon^{\frac{1}{2}}.
  \end{equation*}
\end{cor}
The proof of this corollary builds on the techniques developed for Theorem \ref{thm:main} and is given in detail in Appendix \ref{sec:Linfty}.

\section{Bootstrap}\label{sec:bootstrap}
We now demonstrate the bootstrap argument used to prove Theorem \ref{thm:main}, deferring the details of the relevant estimates for the dispersion, energy and for weights on the profile function to the later sections of this article. The general approach here is a typical continuity argument for dispersive equations and has been used successfully in a plethora of other cases.

We will prove Theorem \ref{thm:main} by showing that for $M\in\N$ there exist $\epsilon_0$ and $k_0$ (depending on $M$) such that for $\epsilon\leq\epsilon_0$ its claim holds for $T\gtrsim\epsilon^{-M}$, which can be made arbitrarily large.

\begin{proof}[Bootstrapping Scheme]
Given $M\in\N$ and $\omega_0$ satisfying the assumptions in Theorem \ref{thm:main} we make the \emph{bootstrapping assumption} that for $0\leq t\leq\epsilon^{-M}$ 
\begin{equation}\label{eq:bstrap1}
 \abs{\omega(t)}_{L^\infty}+\abs{u(t)}_{L^\infty}+\abs{Du(t)}_{L^\infty} \leq \frac{\epsilon^{\frac{1}{8}}}{\ip{t}}.
\end{equation}
From this we will deduce that the energy remains bounded, i.e.
\begin{equation}\label{eq:bstrap_en}
 \abs{\omega(t)}_{H^k}\lesssim \epsilon^{\frac{1}{2}},
\end{equation}
where $k\in\N$ is chosen sufficiently large.

Then we shall go on to conclude the bound for two derivatives on a weight on the profile $f(t)=e^{-L_1t}\omega(t)$
\begin{equation}\label{eq:bstrap_2w}
 \abs{D^2xf(t)}_{L^2}\lesssim \epsilon^{\frac{1}{2}},
\end{equation}
from which it will follow that also three derivatives on a weight on the profile remain under control:
\begin{equation}\label{eq:bstrap_3w}
 \abs{D^3xf(t)}_{L^2}\lesssim \epsilon^{\frac{1}{2}}.
\end{equation}

From this it then follows using the dispersive estimate that the following stronger decay estimate holds:
\begin{equation*}\label{eq:bstrap4}
 \abs{\omega(t)}_{L^\infty}+\abs{u(t)}_{L^\infty}+\abs{Du(t)}_{L^\infty} \lesssim \frac{\epsilon^{\frac{1}{4}}}{\ip{t}}.
\end{equation*}
By a continuity argument this allows us to close our decay, energy and weighted estimates on the time interval $[0,\epsilon^{-M}]$.

The purpose of the present section is to show how these steps \eqref{eq:bstrap_en}, \eqref{eq:bstrap_2w} and \eqref{eq:bstrap_3w} can be derived by combining results established in later sections.


\subsection*{Energy Estimate}
The blow-up criterion \eqref{eq:energy_ineq} in our energy estimate (Lemma \ref{lem:energy_ineq}) gives
 \begin{equation}\label{eq:tobound_en}
  \abs{\omega(t)}_{H^k}\lesssim \epsilon \exp\left(c_12^{2k}\int_0^t \frac{\epsilon^{\frac{1}{8}}}{\ip{s}}ds\right)\lesssim \epsilon  t^{c(k)\epsilon^{\frac{1}{8}}},
 \end{equation}
 where $c(k):=c_1 2^{2k}$. This we aim to bound by $\epsilon^\frac{1}{2}$.

\subsection*{Weighted Estimate}
Combining Propositions \ref{prop:highfreq} and \ref{prop:lowfreq} yields the estimate
 \begin{equation*}
  \begin{aligned}
   \abs{D^2 xf(t)}_{L^2}&\lesssim \epsilon + \int_0^t\frac{\epsilon^{\frac{1}{8}}}{\ip{s}}\left(1+\rho^{-\frac{3}{2}}N^\rho s^\rho\right)\abs{D^2 xf(s)}_{L^2}ds\\
   &\hspace{-1cm} + 2^k \int_0^t \epsilon^{\frac{1}{2}}\frac{\epsilon^{\frac{1}{8}}}{\ip{s}} \left(1+sN^{1-k}+sN^{-3-k}+\rho^{-\frac{3}{2}}N^\rho s^\rho\right)ds\\
   &\hspace{-1cm} + 2^k \int_0^t \epsilon sN^{-2k}ds + \int_0^t \epsilon^{\frac{1}{2}} s \frac{\epsilon^{\frac{1}{4}}}{\ip{s}^2}ds,
  \end{aligned}
 \end{equation*}
with $0<\rho\ll 1$ to be chosen later. Upon choosing $N:=2s^{\frac{1}{k}}$ and carrying out a time integration this reduces to the bound
 \begin{equation*}
 \begin{aligned}
    \abs{D^2 xf(t)}_{L^2}&\lesssim \epsilon +\int_0^t \frac{\epsilon^{\frac{1}{8}}}{\ip{s}}\left(1+\rho^{-\frac{3}{2}} s^{\rho(1+\frac{1}{k})}\right)\abs{D^2 xf(s)}_{L^2}ds\\
    &\quad + k t^{\frac{1}{k}}\epsilon^{\frac{1}{2}}\left(\epsilon^{\frac{1}{2}}+\epsilon^{\frac{1}{8}}(1+\rho^{-\frac{3}{2}} s^{\rho(1+\frac{1}{k})})+\epsilon^{\frac{1}{4}}\right)
 \end{aligned}
 \end{equation*}
An application of Gr\"onwall's inequality then gives
\begin{equation}\label{eq:tobound2w}
 \abs{D^2 xf(t)}_{L^2}\lesssim \left[\epsilon + k t^{\frac{1}{k}}\epsilon^{\frac{1}{2}}\left(\epsilon^{\frac{1}{2}}+\epsilon^{\frac{1}{8}}+\epsilon^{\frac{1}{4}}\right)\right]t^{P(\epsilon,\rho)},
\end{equation}
using that $t<\epsilon^{-M}$ with the exponent $P(\epsilon,\rho)\lesssim\epsilon^{\frac{1}{8}}$ for $\rho$ small enough. This we aim to bound by $\epsilon^{\frac{1}{2}}$.

Assuming this bound on two derivatives and a weight we may use Proposition \ref{prop:added_derivative} in direct analogy to conclude that
 \begin{equation*}
  \begin{aligned}
    \abs{D^3 xf(t)}_{L^2}&\lesssim \epsilon +\int_0^t\frac{\epsilon^{\frac{1}{8}}}{\ip{s}}\left(1+\rho^{-\frac{3}{2}}N^\rho s^\rho\right)\abs{D^3 xf(s)}_{L^2}ds + kt^{\frac{1}{k}}\epsilon^{\frac{1}{8}}\epsilon^{\frac{1}{2}}\\
    &\quad + k t^{\frac{2}{k}}\epsilon^{\frac{1}{2}}\left(\epsilon^{\frac{1}{2}}+\epsilon^{\frac{1}{8}}+\epsilon^{\frac{1}{4}}\right)
  \end{aligned}
 \end{equation*}
where the extra term comes from the estimate on two weights. We thus obtain
\begin{equation}\label{eq:tobound3w}
 \abs{D^3 xf(t)}_{L^2}\lesssim \left[\epsilon + kt^{\frac{1}{k}}\epsilon^{\frac{1}{8}}\epsilon^{\frac{1}{2}}+ k t^{\frac{2}{k}}\epsilon^{\frac{1}{2}}\left(\epsilon^{\frac{1}{2}}+\epsilon^{\frac{1}{8}}+\epsilon^{\frac{1}{4}}\right)\right]t^{P(\epsilon,\rho)}.
\end{equation}
Also this will be bounded by $\epsilon^{\frac{1}{2}}$.

\subsection*{Dispersive Estimate}
To close the bootstrap argument we need to show that from the bounds on the energy and the weights (\eqref{eq:bstrap_en}, \eqref{eq:bstrap_2w} and \eqref{eq:bstrap_3w}) we can actually deduce the stronger decay estimate
\begin{equation*}
 \abs{\omega(t)}_{L^\infty}+\abs{u(t)}_{L^\infty}+\abs{Du(t)}_{L^\infty} \lesssim \frac{\epsilon^{\frac{1}{4}}}{\ip{t}}.
\end{equation*}

To this end we invoke Lemma \ref{lem:disp_est} and deduce that for the choice $N:=2s^{\frac{1}{k}}$ as above we have
\begin{equation*}
\begin{aligned}
 \abs{\omega(t)}_{L^\infty}&=\abs{e^{L_1t}f}_{L^\infty}\lesssim t^{-1} \abs{f}_{H^{3+k}}+t^{-1+\frac{2}{p}} t^{\frac{2}{k}(\mu+\frac{6}{p})} A(\mu)\left(\abs{D^3f}_{L^2}+ \abs{D^3xf}_{L^2}\right)\\
&\lesssim t^{-1}\epsilon^{\frac{1}{2}}+t^{-1+\frac{2}{p}}t^{\frac{2}{k}(\mu+\frac{6}{p})} A(\mu)\epsilon^{\frac{1}{2}}
\end{aligned}
\end{equation*}
with $\frac{1}{p}=1+\mu-\frac{1}{1+\mu}$.

To close we thus need to satisfy the condition 
\begin{equation}\label{eq:tobound_disp}
A(\mu)t^{\frac{2}{p}} t^{\frac{2}{k}(\mu+\frac{6}{p})}\epsilon^{\frac{1}{2}}\leq\epsilon^{\frac{1}{4}}.
\end{equation}

\subsection*{Closing the Bootstrap}
We are now in the position to show how the bootstrap can be closed. This is merely a matter of collecting the conditions established above and showing that they can indeed be satisfied.

Given $M\in\N$ we will establish that for $\epsilon$ small enough the bootstrap \eqref{eq:bstrap1}-\eqref{eq:bstrap_3w} closes on a time interval $[0,\epsilon^{-M}]$ once we choose $k\in\N$ large enough, thus concluding our proof of Theorem \ref{thm:main}.

Firstly, to control the energy and two or three derivatives with a weight on the profile for $0<t\leq \epsilon^{-M}$ we require by \eqref{eq:tobound_en}-\eqref{eq:tobound3w} that
\begin{align}
 &\epsilon \,\epsilon^{-M c(k)\epsilon^{\frac{1}{8}}} \leq\epsilon^{\frac{1}{2}},\label{eq:conden}\\
 &\left[\epsilon + k\epsilon^{-\frac{M}{k}} \epsilon^{\frac{1}{2}}\left(\epsilon^{\frac{1}{2}}+\epsilon^{\frac{1}{8}}+\epsilon^{\frac{1}{4}}\right)\right] \epsilon^{-M\epsilon^{\frac{1}{8}}} \leq \epsilon^{\frac{1}{2}},\label{eq:cond2w}\\
 &\left[\epsilon + k\epsilon^{-\frac{M}{k}}\epsilon^{\frac{1}{8}}\epsilon^{\frac{1}{2}} + k\epsilon^{-\frac{2M}{k}} \epsilon^{\frac{1}{2}}\left(\epsilon^{\frac{1}{2}}+\epsilon^{\frac{1}{8}}+\epsilon^{\frac{1}{4}}\right)\right]\epsilon^{-M\epsilon^{\frac{1}{8}}} \leq \epsilon^{\frac{1}{2}}\label{eq:cond3w}.
\end{align}

Secondly, to close the dispersive estimate \eqref{eq:bstrap1} we require by \eqref{eq:tobound_disp} that
\begin{equation}\label{eq:conddisp}
 A(\mu)\epsilon^{-M\frac{2}{p}} \epsilon^{-M\frac{2}{k}(\mu+\frac{6}{p})}\epsilon^{\frac{1}{2}}\leq\epsilon^{\frac{1}{4}}.
\end{equation}

\noindent\emph{Claim:} There exists a $\lambda>0$ such that for $k=\lambda M$ the above conditions are met.
\begin{proof}
 With this choice we have:
 \begin{enumerate}
  \item Condition \eqref{eq:conden} is equivalent to $\epsilon^{\frac{1}{2}-Mc(k)\epsilon^{\frac{1}{8}}}\leq 1$, which in turn is simply the requirement that $Mc(k)\epsilon^{\frac{1}{8}}\leq \frac{1}{2}$. Clearly for any choice of $M$ and $k=\lambda M$ this can be met by choosing $\epsilon$ small enough (depending on $M$).
  
  \item Conditions \eqref{eq:cond2w} and \eqref{eq:cond3w} will be satisfied if we have $$k\epsilon^{-\frac{2M}{k}}\epsilon^{\frac{5}{8}}\epsilon^{-M\epsilon^\frac{1}{8}}\leq\epsilon^{\frac{1}{2}},$$ i.e.\ if $k\epsilon^{\frac{1}{8}-\frac{2M}{k}-M\epsilon^{\frac{1}{8}}}\leq 1$. For $k>32M$ this will be satisfied once $k\epsilon^{\frac{1}{16}-M\epsilon^{\frac{1}{8}}}\leq 1$, which can be seen to hold for $\epsilon$ small enough.
  
  \item To satisfy condition \eqref{eq:conddisp} we need that $A(\mu)\epsilon^{\frac{1}{4}-2M(\frac{\mu}{k}+\frac{1}{p}+\frac{6}{kp})}\leq 1$. Hence up to the constant $A(\mu)$ we require $\frac{1}{4}\geq \frac{2\mu}{\lambda}+\frac{2M}{p}+\frac{12M}{\lambda p}$. One checks directly that this can be satisfied by choosing $\mu$ (and thus also $\frac{1}{p}$) small enough.
 \end{enumerate}
\end{proof}

This completes the proof of Theorem \ref{thm:main}.
\end{proof}

\section{Energy Estimates}\label{sec:en_ineq}
To obtain energy estimates we differentiate the equation and get the following standard blow-up criterion:
\begin{lemma}\label{lem:energy_ineq}
 Let $\omega$ be a solution of the $\beta$-plane equation \eqref{eq:b-plane} on a time interval containing $[0,T]$. Then for any $k\in\N$ we have the bound
 \begin{equation}\label{eq:energy_ineq}
  \abs{\omega(T)}_{H^k}\leq \abs{\omega(0)}_{H^k} \exp\left(c_1\; 2^{2k}\int_{0}^T \abs{Du(t)}_{L^\infty}+\abs{\omega(t)}_{L^\infty}dt\right)
 \end{equation}
 for some universal constant $c_1$ independent of $k$.
\end{lemma}

Since our arguments will require a large number $k$ of derivatives, here it is important to keep track of the size of the constants with respect to $k$.

\begin{proof}
 In this proof, all implicit constants are assumed to be uniform in $k$.
 
 For any $0\leq j\leq k$, $j\in\N_0$, we let $\alpha\in\N_0^2$ be a multiindex of order $\abs{\alpha}=j$ and proceed by differentiating the equation by $\partial^\alpha$, multiply by $\partial^\alpha\omega$ and integrate over $\R^2$ to obtain 
  $$\partial_t \ip{\partial^\alpha\omega(t),\partial^\alpha\omega(t)}_{L^2}+\ip{\partial^\alpha(u\cdot\nabla\omega),\partial^\alpha\omega}_{L^2}=\ip{L_1 \partial^\alpha\omega,\partial^\alpha\omega}=0.$$
 We note that $\div(u)=0$, so that in estimating $\ip{\partial^\alpha(u\cdot\nabla\omega),\partial^\alpha\omega}_{L^2}$ we only have to bound terms of the form $\abs{\partial^{\beta}u\ccdot\nabla \partial^{\gamma}\omega}_{L^2}$, where $\abs{\beta}=l+1$, $\abs{\gamma}=m$ and $m+l+1=j$. However, we point out that there are up to $2^j$ such terms. To bound one of them we use H\"older's inequality and Gagliardo-Nirenberg interpolation to conclude then that
 \begin{equation*}
  \begin{aligned}
   \abs{\partial^\beta u\ccdot\nabla \partial^\gamma\omega}_{L^2}&\leq\abs{\partial^\beta u}_{L^p}\abs{\nabla \partial^\gamma\omega}_{L^q}\qquad\textnormal{ where }\frac{1}{p}+\frac{1}{q}=\frac{1}{2}\\
   &\lesssim 2^{2j}\abs{Du}_{L^\infty}^\theta\abs{Du}_{\dot{H}^j}^{1-\theta}\abs{\omega}_{L^\infty}^\vartheta \abs{\omega}_{\dot{H}^j}^{1-\vartheta}\\
   &\lesssim 2^{2j}\abs{Du}_{L^\infty}^\theta\abs{\omega}_{L^\infty}^\vartheta\abs{\omega}_{\dot{H}^j}\\
   &\lesssim 2^{2j}\left(\abs{Du}_{L^\infty}+\abs{\omega}_{L^\infty}\right)\abs{\omega}_{\dot{H}^j},
  \end{aligned}
 \end{equation*}
 since the conditions $l-\frac{2}{p}=(1-\theta)(j-1)$ and $m+1-\frac{2}{q}=(1-\vartheta)(j-1)$ imply that $\theta+\vartheta=1=1-\theta+1-\vartheta$.

 Thus
  $$\partial_t\abs{\omega(t)}_{\dot{H}^j}\lesssim 2^{2 j}\left(\abs{Du(t)}_{L^\infty}+\abs{\omega(t)}_{L^\infty}\right)\abs{\omega(t)}_{\dot{H}^j},$$
 from which summation over $j\leq k$ and Gr\"onwall's Lemma give \eqref{eq:energy_ineq}.
\end{proof}

\section{Dispersive Estimate}\label{sec:disp_est}
To understand the decay properties of the semigroup $e^{L_1t}$ we adopt the point of view of oscillatory integrals. We write
\begin{equation*}
 e^{L_1 t}g(x)=\int_{\R^2}e^{ix\cdot\xi -it\frac{\xi_1}{\abs{\xi}^2}}\hat{g}(\xi)\,d\xi, \quad g\in C_c^\infty(\R^2).
\end{equation*}

Asymptotics for such integrals have been studied extensively (e.g.\ in \cite{MR1996773}). Since the Hessian of the phase $it\left(\frac{x}{t}\cdot\xi -\frac{\xi_1}{\abs{\xi}^2}\right)$ of this integral is non-degenerate, the proof of the following proposition is a standard argument using stationary phase and Littlewood-Paley techniques and has already been hinted at in our previous work \cite{SQGpaper} -- we include it in Appendix \ref{sec:disp_proof} for completeness' sake. The need to control three derivatives of $g$ can be seen by a scaling argument.

\begin{prop}\label{prop:disp_est}
 There is a constant $C>0$ such that for any $g\in C^\infty_c(\R^2)$ we have\footnote{Here $\dot{B}^s_{p,q}$ denotes the homogeneous Besov space with $s$ derivatives of the frequency localized pieces in $L^p$, summed in $\ell^q$ and norm $\abs{\psi}_{\dot{B}^s_{p,q}}:=\abs{\left(2^{sj}\abs{P_j\psi}_{L^p}\right)_{j\in Z}}_{\ell^q}$, $s\in\R$, $1\leq p,q\leq\infty$.}
 \begin{equation}\label{eq:disp_est}
  \abs{e^{L_1 t}g}_{L^\infty}\leq C\, t^{-1}\abs{g}_{\dot{B}^{3}_{1,1}}.
 \end{equation}
\end{prop}

More importantly, the goal of this section is to show how we can employ the weighted estimates to obtain decay of $\omega$ in $L^\infty$. The main issue that arises is that the control of one weight in $L^2$ does not give $L^1$ integrability, i.e.\ $\abs{xg}_{L^2}$ does not control $\abs{g}_{L^1}$. Thus we cannot use \eqref{eq:disp_est} directly. However, as will be clear from the details in Sections \ref{sec:weights}ff., it is not an easy matter to obtain other weighted estimates that do not grow fast, so we will work with only one weight and accept the loss this entails. The idea here is to avoid $L^\infty$ estimates either using energy estimates for high frequencies, or $L^p$ spaces with large $p$ for low frequencies.

\begin{lemma}\label{lem:disp_est}
 For any $0<\mu\ll 1$ and $k\in\N$ we have
 \begin{equation}
  \abs{e^{L_1t}f}_{L^\infty}\lesssim 2^k N^{-k}\abs{f}_{H^{3+k}}+t^{-1+\frac{2}{p}} N^{2\mu+\frac{12}{p}} A(\mu)\left(\abs{D^3f}_{L^2}+ \abs{D^3xf}_{L^2}\right)
 \end{equation}
 where $\frac{1}{p}=1+\mu-\frac{1}{1+\mu}$ and $N>0$ can be chosen later.
\end{lemma}

\begin{remark}\label{rem:u_Du_decay}
 The fact that these bounds employ $L^2$ based spaces shows that the above arguments can equally be applied to Riesz transforms of $f$ to obtain their decay under the action of the semigroup $e^{L_1t}$. In particular, this lemma shows that we can control $\omega$ in $L^\infty$ using one weight with \emph{three} derivatives in $L^2$ (and an energy term). The velocity $u$ on the other hand is one degree smoother and contains a Riesz transform: Recall that $u(t)=e^{L_1t}\nabla^\perp(-\Delta)^{-1}f(t)$, so this can be controlled using one weight with only \emph{two} derivatives.
 
 As we will see in Sections \ref{sec:weights}-\ref{sec:res}, our proof of Theorem \ref{thm:main} first establishes bounds for two derivatives on a weight of the profile $f$, which then gives bounds for three derivatives on a weight of $f$. In terms of the variables of the equation, we first obtain control of the decay of $u$, which then allows us to control that of $\omega$.
\end{remark}

\begin{proof}[Proof of Lemma \ref{lem:disp_est}]
Firstly we note that for high frequencies we may invoke the energy estimates via \eqref{eq:hfreqgain} and obtain
\begin{equation*}
 \abs{P_{>N}f}_{L^\infty}\lesssim 2^k N^{-k}\abs{f}_{H^{3+k}},
\end{equation*}
so that we may restrict to frequencies less than or equal to some $N>0$.

Next we notice that by Bernstein's inequality in two space dimensions we have for any $\mu>0$
\begin{equation*}
 \abs{P_{\leq N}f}_{L^\infty}\lesssim N^{2\mu} \abs{P_{\leq N}f}_{L^{1/\mu}},
\end{equation*}
allowing us to work with lower integrabilities than infinity. As for the dispersive estimate we write
\begin{equation*}
 \begin{aligned}
  \abs{P_{\leq N}e^{L_1t}f}_{L^r}&\sim \sum_{2^j\leq N} \abs{e^{L_1t}P_jf}_{L^r} \sim\sum_{2^j\leq N}\abs{e^{L_1t}\check{\varphi_j}\ast P_jf}_{L^r}\\
  &\leq \sum_{2^j\leq N} \abs{e^{L_1t}\check{\varphi_j}}_{L^p}\abs{P_jf}_{L^q}
 \end{aligned}
\end{equation*}
as long as $1+\frac{1}{r}=\frac{1}{p}+\frac{1}{q}$. A direct computation shows that
\begin{equation*}
 \abs{e^{L_1t}\check{\varphi_j}}_{L^p}=2^{2j}\abs{e^{L_1 2^{-j}t}\check{\varphi}(2^j x)}_{L^p}=2^{2j}2^{-j\frac{2}{p}}\abs{e^{L_1 2^{-j}t}\check{\varphi}}_{L^p}.
\end{equation*}
Furthermore we may interpolate between $\abs{e^{L_1t}\check{\varphi}}_{L^2}\sim 1$ and $\abs{e^{L_1t}\check{\varphi}}_{L^\infty}\sim t^{-1}$ (see Proposition \ref{prop:disp_est}) to conclude that
\begin{equation*}
 \abs{e^{L_1t}\check{\varphi}}_{L^p}\lesssim t^{-1+\frac{2}{p}}.
\end{equation*}

Combining this with the above estimates when $r=\frac{1}{\mu}$ and $q=1+\mu$ yields
\begin{equation}
 \begin{aligned}
  \abs{e^{L_1t}P_{\leq N}f}_{L^\infty}&\lesssim N^{2\mu}\abs{e^{L_1t}P_{\leq N}f}_{L^{1/\mu}}\\
  &\lesssim t^{-1+\frac{2}{p}} N^{2\mu} \sum_{2^j\leq N} 2^{3j}2^{-j\frac{4}{p}}\abs{P_jf}_{L^{1+\mu}}
 \end{aligned}
\end{equation}
with $\frac{1}{p}=1+\mu-\frac{1}{1+\mu}$.

Now we use that by H\"older's inequality we can control $L^{1+\mu}$ by $L^2$ near the origin and by one weight in $L^2$ near infinity. More precisely, we have the estimates
\begin{equation*}
 \abs{f\chi_{\abs{x}\leq R}}_{L^{1+\mu}}\leq \abs{f}_{L^2} \left(\frac{R}{\sqrt{2}}\right)^{\frac{1-\mu}{1+\mu}}
\end{equation*}
and
\begin{equation*}
 \abs{f\chi_{\abs{x}> R}}_{L^{1+\mu}}\leq \left[\frac{1-\mu}{4\mu}R^{-\frac{4\mu}{1-\mu}}\right]^{\frac{1-\mu}{2(1+\mu)}}  \abs{xf}_{L^2}.
\end{equation*}
Optimizing $R$ with respect to $\mu$ suggests the choice $R:=\mu^{-\frac{1-\mu}{2(1+\mu)}}$, which leaves us with
\begin{equation}
 \abs{f}_{L^{1+\mu}}\leq A(\mu)\left(\abs{f}_{L^2}+\abs{xf}_{L^2}\right),\quad A(\mu)\sim\mu^{-\frac{(1-\mu)^2}{2(1+\mu)^2}}.
\end{equation}
In particular, notice that $A(\mu)\rightarrow\infty$ as $\mu\rightarrow 0$.

Returning to the above dispersive estimate we thus see that
\begin{equation}
 \begin{aligned}
  \abs{e^{L_1t}P_{\leq N}f}_{L^\infty}&\lesssim t^{-1+\frac{2}{p}} N^{2\mu} \sum_{2^j\leq N}2^{j(3-\frac{4}{p})}\abs{P_jf}_{L^{1+\mu}}\\
  &\leq t^{-1+\frac{2}{p}} N^{2\mu}A(\mu)\sum_{2^j\leq N}2^{j(3-\frac{4}{p})}\left(\abs{P_j f}_{L^2}+ \abs{xP_jf}_{L^2}\right)\\
  &\lesssim t^{-1+\frac{2}{p}} N^{2\mu} A(\mu)\left(\abs{P_{\leq N}D^{3-\frac{4}{p}}f}_{L^2}+ \abs{P_{\leq N}D^{3-\frac{4}{p}}xf}_{L^2}\right)\\
  &\lesssim t^{-1+\frac{2}{p}} N^{2\mu} A(\mu) N^{\frac{12}{p}}\left(\abs{P_{\leq N}D^3f}_{L^2}+ \abs{P_{\leq N}D^3xf}_{L^2}\right).
 \end{aligned}
\end{equation}
Here the last inequality holds since one can combine the inequalities of H\"older and Bernstein to show that in $\R^2$ one has $\abs{P_{\leq N}D^{-\nu}f}_{L^2}\lesssim N^{3\nu}\abs{P_{\leq N}f}_{L^2}$ for $0<\nu<1$. 
\end{proof}

\section{Towards the Weighted Estimate}\label{sec:weights}
Our goal for this and the following two sections is to provide $L^2$ bounds for a weight with two or three derivatives on the profile function, i.e.\ bounds for $D^2xf$ and $D^3xf$ in $L^2$. This is done first for two derivatives (Propositions \ref{prop:lowfreq} and \ref{prop:lowfreq}), from which we can deduce estimates for the case of three derivatives (Proposition \ref{prop:added_derivative}). We thus obtain first a bound for $u$, which then yields a bound for $\omega$ (see also Remark \ref{rem:u_Du_decay}).

For this we will essentially use estimates of the type $L^2_x\times L^\infty_x\to L^2_x$. Informally speaking, in terms of our bootstrapping scheme in Section \ref{sec:bootstrap}, every bilinear term then contains a decay factor of $t^{-1}$ through dispersion (see \eqref{eq:disp_est}) and gives back the weight in $L^2$ or an energy term, both of which are assumed to remain small.

We start with a general discussion of our approach and give a brief overview of the structure of the remaining chapters and the difficulties we encounter.

\subsection{The Duhamel Term and Remarks on the Present Approach}\label{sec:DHTerm}
The Duhamel formula for the $\beta$-plane equation on the profile $f(t):=e^{-L_1t}\omega(t)$ in Fourier space reads
\begin{equation}\label{eq:DH}
 \hat{f}(t,\xi)=\hat{f_0}(\xi)+\int_0^t\int_{\R^2}e^{is\Phi}m(\xi,\eta)\hat{f}(s,\xi-\eta)\hat{f}(s,\eta)d\eta ds,
\end{equation}
where the phase $\Phi$ is given as $$\Phi=\frac{\xi_1}{\abs{\xi}^2}-\frac{\xi_1-\eta_1}{\abs{\xi-\eta}^2}-\frac{\eta_1}{\abs{\eta}^2}$$ and the multiplier is 
$$m(\xi,\eta)=\frac{\xi\ccdot\eta^\perp}{\abs{\eta}^2}$$
or (by changing variables $\eta\leftrightarrow\xi-\eta$) 
$$\bar{m}(\xi,\eta):=m(\xi,\xi-\eta)=-\frac{\xi\cdot\eta^\perp}{\abs{\xi-\eta}^2}.$$ 
We note that these have a \emph{null structure} which annihilates waves with parallel frequencies.


As is clear from the Duhamel formula \eqref{eq:DH}, all estimates on $\hat{f_0}$ just amount to assumptions on the initial data, so we may restrict ourselves to bounding the time integral term
$$\abs{\abs{\xi}^l\partial_\xi\int_0^t\int_{\R^2}e^{is\Phi}m(\xi,\eta)\hat{f}(s,\xi-\eta)\hat{f}(s,\eta)d\eta ds}_{L^2_\xi}$$
for $l=2,3$.

Up to powers of $\abs{\xi}$ this leaves us with the following three main types of terms\footnote{In what follows we may drop the explicit dependence of $\hat{f}$ on time from our notation.} to be dealt with:
\begin{align}
 &\int_0^t\int_{\R^2}\partial_\xi\left(e^{is\Phi}\right)m(\xi,\eta)\hat{f}(\xi-\eta)\hat{f}(\eta)d\eta ds,\label{eq:hard_piece}\\
 &\int_0^t\int_{\R^2}e^{is\Phi}\partial_\xi m(\xi,\eta)\hat{f}(\xi-\eta)\hat{f}(\eta)d\eta ds,\label{eq:dmult_term}\\
 &\int_0^t\int_{\R^2}e^{is\Phi}m(\xi,\eta)\partial_\xi\hat{f}(\xi-\eta)\hat{f}(\eta)d\eta ds.\label{eq:direct_weight}
\end{align}

Of these, \eqref{eq:hard_piece} is the most difficult one due to an extra factor of $s$ from the differentiation $\partial_\xi\left(e^{is\Phi}\right)=is\partial_\xi\Phi e^{is\Phi}$. For \eqref{eq:dmult_term} we may just invoke the energy estimates (see Section \ref{sec:dmult_term}), while \eqref{eq:direct_weight} gives back the weight (once we notice that we may assume $\xi$ and $\xi-\eta$ to be of comparable size, as remarked in Section \ref{sec:simpl}). In particular we note that for these two terms the cases of two or three derivatives are completely analogous. It is the difficulties encountered in \eqref{eq:hard_piece} that guide the structure of the following chapters.

More precisely, for \eqref{eq:hard_piece} a detailed analysis of the resonance structure of the equation is needed. This in turn gives rise to Fourier multipliers in bilinear terms. However, due to the anisotropy and singularity of the linear propagator, the asymptotics of these do not fall into the classical categories of well-studied multipliers. In particular, the authors are not aware of any techniques that would allow to directly treat the bilinear terms that so arise in a satisfactory way.

We conclude this section with some remarks regarding the particulary difficult terms of the type of \eqref{eq:hard_piece} and useful techniques for the bilinear estimates.
\subsubsection{Difficulties in \eqref{eq:hard_piece}}\label{sec:difficulties}
As we will see through the reductions in Section \ref{sec:weights_red}, this term is the heart of the matter and will be discussed in detail in Section \ref{sec:res}. The basic issue is that we need to control the additional factor $s$ in the time integration:
\begin{equation*}
\begin{aligned}
 \int_0^t\int_{\R^2}&\abs{\xi}^2\partial_\xi\left(e^{is\Phi}\right)m(\xi,\eta)\hat{f}(\xi-\eta)\hat{f}(\eta)d\eta ds\\
 &=\int_0^t\int_{\R^2}is\abs{\xi}^2\partial_\xi\Phi e^{is\Phi}m(\xi,\eta)\hat{f}(\xi-\eta)\hat{f}(\eta)d\eta ds.
\end{aligned}
\end{equation*}

As is well known, the typical ways of dealing with this are integrations by parts in either space or time (i.e.\ exploiting the lack of space or time resonances in various domains of Fourier space). In our case the anisotropy of the equation and the degeneracy/singularity of the linear propagator make this a difficult task. In particular, we cannot even calculate the time resonances explicitly.

Up to singularities when $\xi=0$, $\eta=0$ or $\xi=\eta$ we have:
\begin{itemize}
 \item Time resonances $\mathcal{T}:=\{(\xi,\eta)\in\R^2\times\R^2: \Phi(\xi,\eta)=0\}$: They contain the sets $\{\xi_1=\eta_1=0\}$ and $\{\xi_1=\eta_1, \xi_2=\pm\eta_2\}$, for example, but their explicit computation is practically infeasible.
 \item Space resonances $\mathcal{S}:=\{(\xi,\eta)\in\R^2\times\R^2: \nabla_\eta\Phi(\xi,\eta)=0\}$: These are given by $\{\xi=2\eta\}$.
 \item Spacetime resonances $\mathcal{R}:=\mathcal{S}\cap\mathcal{T}=\{(\xi,\eta)\in\R^2\times\R^2: \Phi(\xi,\eta)=\nabla_\eta\Phi(\xi,\eta)=0\}$: These are given by $\{(\xi=(0,2\lambda),\eta=(0,\lambda):\;\lambda\neq 0\}$.
\end{itemize}

Amongst other points, through Section \ref{sec:weights_red} we will see how we can reduce our study of this term to considering only ``low'' frequencies. Section \ref{sec:res} then explains how we obtain the relevant estimates through a detailed resonance analysis. More precisely, Sections \ref{sec:direct_approach} and \ref{sec:time_res} lay out the available techniques for this, for which more details are given in Section \ref{sec:splitting}.

\subsubsection{A Remark on the Divergence Structure of the Equation}\label{sec:div_structure}
Although for most of what follows we will work with the above Duhamel formulation \eqref{eq:DH} of the $\beta$-plane equation in Fourier space and use techniques inspired by the treatment of oscillatory integrals, it is useful to keep in mind the divergence structure of the equation: Recall that we have $\div(u)=0$. In our calculations for the weighted estimate $\abs{D^2xf}_{L^2}$ we will encounter terms of the form $e^{-L_1t}\left(u\ccdot\nabla(e^{L_1t}D^2xf)\right)$. It is useful to write these in physical space rather than Fourier space, since this way it is easier to see that a cancellation occurs for the highest order of derivatives on one of the profiles, i.e.\
\begin{equation*}
 \ip{e^{-L_1t}\left(u\ccdot\nabla(e^{L_1t}D^2xf)\right),D^2xf}=\ip{u\ccdot\nabla (e^{L_1t}D^2xf),e^{L_1t}D^2xf}=0.
\end{equation*}

Taking into account possible localizations this shows that in such cases the null form does not add further derivatives. In particular, for the case of the weight with derivatives falling onto one of the profiles directly as in \eqref{eq:direct_weight}, we will not lose any derivatives on the weight. This is important to note, for otherwise we could not close the estimates for the weight with a fixed amount of derivatives.

\subsubsection{A Remark on Bilinear Estimates}\label{sec:bilinear_remark}
The basic estimate for the bilinear expressions that we will encounter is of type $L^2_x\times L^\infty_x\to L^2_x$. Such an estimate is easily available if the Fourier multipliers factor onto the profiles (see Example \ref{ex:bilinear} below). In the more complicated cases we will rely one of the following theorems below. 

The classical estimate by Coifman and Meyer, which gives bounds for smooth bilinear symbols with the appropriate decay assumptions (see \cite{CM1}, \cite{CM2}), will be useful when the inputs in the Duhamel formula \eqref{eq:DH} have comparable frequencies:
\begin{thm}[Coifman-Meyer]\label{thm:CM}
 Let $b\in L^\infty(\R^2\times\R^2)$ be smooth away from the origin and define the associated bilinear operator through
  \begin{equation*}
   \widehat{T_b(f,g)}(\xi):=\int_{\R^2} b(\xi,\eta)\hat{f}(\xi-\eta)\hat{g}(\eta)\;d\eta.
  \end{equation*}
 Assume moreover that for any multi-indices $\alpha,\beta\in\N_0^2$ with $\abs{\alpha}+\abs{\beta}\leq 5$ we have for some $C>0$
  \begin{equation}
   \abs{\partial_\xi^\alpha \partial_\eta^\beta b(\xi,\eta)}\leq C\left(\abs{\xi}+\abs{\eta}\right)^{-\abs{\alpha}-\abs{\beta}}.
  \end{equation}
 Then $T_b$ is a continuous bilinear map $T_b:L^p\times L^q\rightarrow L^r$ for $1<p,q,r\leq\infty$ satisfying $\frac{1}{r}=\frac{1}{p}+\frac{1}{q}$, i.e.\ there exists a constant $K=K(p,q,r,C)$ such that
  \begin{equation}
   \abs{T_b(f,g)}_{L^r}\leq K \abs{f}_{L^p}\abs{g}_{L^q}.
  \end{equation}
\end{thm}

As noted before, the relevant setup here are estimates of the type $L^2_x\times L^\infty_x\to L^2_x$. On the large scale the idea is then to make use of the $L^\infty$ decay properties of the semigroup $e^{L_1t}$ (through the dispersive estimate \eqref{eq:disp_est}) to control the time integrals in the Duhamel term \eqref{eq:DH} and its versions.

In the case of a localization to a domain in frequency space where one of the input frequencies in \eqref{eq:DH} is much smaller than the other we shall make use of the following result, as presented in \cite{MR2775116} (and also contained in \cite{MR2559713}):
\begin{thm}\label{thm:multhm}
 Assume that for $0\leq s\leq 1$ the function $b:\R^2\times\R^2\to\R$ satisfies
 \begin{equation*}
  \sum_{j\in\Z}\abs{P_j^\eta b(\xi,\eta)}_{L^\infty_\xi \dot{H}^s_\eta}=:\abs{b}_{\mathcal{M}^s}<\infty,
 \end{equation*}
 where the Littlewood-Paley projections $P_j^\eta$ act in the variable $\eta$. Then the associated bilinear multiplier
 \begin{equation*}
  \widehat{T_b(f,g)}(\xi):=\int_{\R^2} b(\xi,\eta)\hat{f}(\xi-\eta)\hat{g}(\eta)\;d\eta
 \end{equation*}
 is continuous with norm $\abs{b}_{\mathcal{M}^s}$ as a map $T_b:L^a\times L^2\to L^{c'}$ for $a,c$ satisfying
 \begin{equation*}
  2\leq a,c\leq \frac{2}{1-s},\quad \frac{1}{a}+\frac{1}{c}=1-\frac{s}{2}
 \end{equation*}
 and $1=\frac{1}{c}+\frac{1}{c'}$.
\end{thm}

As it turns out, the multipliers we encounter in this setup are exactly at the borderline $s=1$, $a=\infty$ and $c=c'=2$ in terms of their singularities at the origin.

\begin{example}\label{ex:bilinear}
The obvious $L^2_x$ estimate for the bilinear part of the Duhamel term just amounts to the physical space estimate $\abs{u\cdot\nabla\omega}_{L^2_x}$. We give it here starting from the Fourier expression \eqref{eq:DH} to illustrate the general procedure which will later be employed for similar terms. We note that the null form and the phase simply factor on the profiles and on the integral, so that from the perspective of the Coifman-Meyer Theorem \ref{thm:CM} we are just dealing with a bilinear multiplier with symbol $1$. Thus
\begin{equation*}
 \begin{aligned}
  &\int_0^t\int_{\R^2}e^{is\Phi}m(\xi,\eta)\hat{f}(s,\xi-\eta)\hat{f}(s,\eta)d\eta ds\\
  &\quad =\int_0^t e^{is\frac{\xi_1}{\abs{\xi}^2}}\int_{\R^2}(\xi-\eta)e^{-is\frac{\xi_1-\eta_1}{\abs{\xi-\eta}^2}}\hat{f}(s,\xi-\eta)\cdot \frac{\eta^\perp}{\abs{\eta}^2}e^{-is\frac{\eta_1}{\abs{\eta}^2}}\hat{f}(s,\eta)d\eta ds\\
  &\quad =\int_0^t  e^{is\frac{\xi_1}{\abs{\xi}^2}}\left(\int_{\R^2}\mathcal{F}(\nabla e^{L_1s}f)(s,\xi-\eta) \cdot \mathcal{F}(\nabla^\perp(1-\Delta)^{-1} e^{L_1s}f)(s,\eta) d\eta\right)ds\\
  &\quad =\int_0^t e^{is\frac{\xi_1}{\abs{\xi}^2}} \mathcal{F}\left(u\ccdot\nabla\omega(s)\right)ds
 \end{aligned}
\end{equation*}
and we conclude that
\begin{equation*}
\begin{aligned}
 &\abs{\int_0^t\int_{\R^2}e^{is\Phi}m(\xi,\eta)\hat{f}(s,\xi-\eta)\hat{f}(s,\eta)d\eta ds}_{L^2_\xi}\\
 &\quad \leq \int_0^t \abs{ e^{is\frac{\xi_1}{\abs{\xi}^2}} \mathcal{F}\left(u\ccdot\nabla\omega(s)\right)}_{L^2_\xi}ds\\
 &\quad =\int_0^t \abs{ e^{-sL_1}\left(u\ccdot\nabla\omega(s)\right)}_{L^2_x}ds\\
 &\quad =\int_0^t \abs{ u\ccdot\nabla\omega(s)}_{L^2_x}ds
\end{aligned}
\end{equation*}
by unitarity of the linear semigroup.
\end{example}

In general we will follow a similar line of reasoning, the significant difference being the presence of more complicated multipliers. As we remarked above, because of the structure of the multipliers that arise we will only follow this strategy for frequencies that are bounded by some $N>0$ (depending on time -- the details being made explicit later). Hence one needs to check that the various multipliers have no singularities in the respective localizations, for one can then see that the conditions of the Coifman-Meyer multiplier theorem are met by the homogeneity of the terms involved.

\begin{notation}\label{rmk:CM_not} For a bilinear Fourier multiplier $a(\xi,\eta)$ we shall write $$a(\xi,\eta) \lesssim_{CM} \abs{\xi}^\alpha \abs{\eta}^\beta \abs{\xi-\eta}^\gamma$$ to mean that $\frac{a(\xi,\eta)}{\abs{\xi}^\alpha \abs{\eta}^\beta \abs{\xi-\eta}^\gamma}$ satisfies the conditions of the Coifman-Meyer multiplier theorem \ref{thm:CM}, where $\alpha,\beta,\gamma\in\N$. The expression $\abs{\xi}^\alpha \abs{\eta}^\beta \abs{\xi-\eta}^\gamma$ then factors onto the profiles and/or gives derivatives in the physical space variable $x$ corresponding to $\xi$.

\emph{In most cases we only explicitly state the computations required for the verification of the Coifman-Meyer conditions up to first order. Because of the homogeneities of the involved terms the further computations needed to verify the applicability of Coifman-Meyer's theorem are tedious rather than difficult and are hence not always fully detailed.}
\end{notation}

\section{Weighted Estimate -- Reductions}\label{sec:weights_red}
The goal for the rest of the paper is to prove the following Proposition (and its precursors, Propositions \ref{prop:highfreq} and \ref{prop:lowfreq}) regarding the weighted estimate on the profile $f$ of $\omega$:
\begin{prop}\label{prop:added_derivative}
 If $\omega$ is a solution of the $\beta$-plane equation \eqref{eq:IVP} and $f(t):=e^{-L_1t}\omega(t)$ we have the bounds
  \begin{equation}\label{eq:3hiweights}
  \begin{aligned}
   \abs{P_{>N}D^3 xf(t)}_{L^2}&\lesssim \abs{P_{>N}D^3 xf_0}_{L^2}+ \int_0^t \abs{\omega(s)}_{L^\infty} \abs{D^3 xf(s)}_{L^2}ds\\
   &\hspace{-1cm} + 2^k \int_0^t \abs{\omega(s)}_{H^{4+k}}\left(\abs{\omega(s)}_{L^\infty}+\abs{u(s)}_{L^\infty}\right) \left(1+sN^{1-k}+sN^{-3-k}\right)ds\\
   &\hspace{-1cm} + 2^k\int_0^t sN^{-2k}\abs{\omega(s)}_{H^{5+k}}^2 ds
  \end{aligned}
  \end{equation}
  and
  \begin{equation}\label{eq:3weights}
     \begin{aligned}
    \abs{P_{\leq N}D^3xf(t)}_{L^2}&\lesssim \abs{P_{\leq N}D^3 xf_0}_{L^2}+\int_0^t N\abs{D^2xf(s)}_{L^2}\abs{\omega(s)}_{L^\infty}ds\\
    &\quad +\int_0^t \abs{\omega(s)}_{L^\infty} \abs{D^3xf(s)}_{L^2} ds\\
    &\quad +\int_0^t N\left(\abs{\omega(s)}_{L^\infty}+\abs{u(s)}_{L^\infty}\right)\abs{\omega(s)}_{H^1}ds\\
    &\quad +\int_0^t sN \left(\abs{u(s)}^2_{L^\infty}+\abs{u(s)}_{L^\infty}\abs{Du(s)}_{L^\infty}\right)\abs{\omega(s)}_{H^2}ds\\
    &\quad +\int_0^t \rho^{-\frac{3}{2}}N^{\rho}\abs{D^3xf(s)}_{L^2}\abs{\omega(s)}_{L^{\frac{2}{\rho}}}ds\\
    &\quad +\int_0^t \rho^{-\frac{3}{2}}N^{1+\rho}\abs{D^2xf(s)}_{L^2}\abs{\omega(s)}_{L^{\frac{2}{\rho}}}ds\\
    &\quad +\int_0^t \rho^{-\frac{3}{2}}N^{1+\rho}\left(\abs{\omega(s)}_{L^\frac{2}{\rho}}+\abs{u(s)}_{L^\frac{2}{\rho}}\right)\abs{\omega(s)}_{H^1}ds,
  \end{aligned}
  \end{equation}
 where $0<\rho\ll 1$.
\end{prop}

\begin{proof}[Outline of the Proofs of Propositions \ref{prop:added_derivative}-\ref{prop:lowfreq}]
At first we focus on proving estimates for two derivatives with a weight on the profile function. This is the main difficulty, from which the case of an additional third derivative on the weighted profile function follows -- we explain this at the very end, in Section \ref{sec:add_deriv}.

After a useful observation in Section \ref{sec:simpl}, for clarity we present the estimates for simpler terms \eqref{eq:dmult_term} and \eqref{eq:direct_weight} (Sections \ref{sec:dmult_term} and \ref{sec:direct_weight}, respectively), which we will rely on later.

Subsequently we will treat the cases of ``high'' and ``low'' frequencies separately. The idea is that when restricting -- in a time-dependent manner -- to high enough frequencies we can compensate for faster time growth of some of the terms through the boundedness of many derivatives through the Sobolev embedding -- this is the content of Section \ref{sec:highfreq} and includes the proof of Proposition \ref{prop:highfreq}. In particular, no further manipulations will then be necessary and all multipliers directly factor onto the various terms, so that no special multiplier theorems are needed.

We have thus reduced to the study of the term \eqref{eq:hard_piece} in the regime of bounded ``low'' frequencies: This is the focuse of Section \ref{sec:res}. Here we have to make use of the resonance and null structures in the equation to avoid too much growth in time. This leads us to split up Fourier space into several regions, where distinct resonance effects dominate. The variety of multipliers thus produced (essentially by integration by parts) satisfy the requirements of one of the multiplier theorems cited above in Remark \ref{sec:bilinear_remark} due to our localizations and our having avoided singularities of the multipliers involved. Due to the homogeneities of their constituent pieces, the conditions in the various localizations will not be too difficult to check.

Before giving the full details in exactly the context we need them in, we lay out the main ideas: For low frequencies we will combine versions of a space resonance approach (Section \ref{sec:direct_approach}) as well as time resonance methods (Section \ref{sec:time_res}) for different regions of frequency space. The details of the splitting of Fourier space into appropriate subdomains and the corresponding estimates are then given in Section \ref{sec:splitting}.

The results are collected in Propositions \ref{prop:highfreq} and \ref{prop:lowfreq}, which give bounds for high and low frequencies, respectively, with two derivatives and a weight on the profile. Finally, Proposition \ref{prop:added_derivative} contains the estimates for three derivatives and a weight on the profile. This concludes the proofs for the weighted estimates.
\end{proof}

\subsection{A First Observation}\label{sec:simpl}
We remark that due to the invariance of $\Phi$ under a change of variables $\eta\leftrightarrow\xi-\eta$ we can always arrange that $\abs{\eta}\lesssim\abs{\xi-\eta}$ (which entails $\abs{\xi}\sim\abs{\xi-\eta}$). Indeed, let $\zeta:\R^+\to\R$ be a smooth cutoff function with $\supp(\zeta)\subset [0,2]$ and $\zeta|_{[0,1]}=1$. Now we split the space integral in the Duhamel formula \eqref{eq:DH} by writing
\begin{equation*}
\begin{aligned}
 \int_{\R^2}e^{is\Phi}m(\xi,\eta)\hat{f}(\xi-\eta)\hat{f}(\eta)d\eta &= \int_{\R^2}e^{is\Phi}m(\xi,\eta)(1-\zeta)\left(\frac{\abs{\xi-\eta}}{\abs{\eta}}\right)\hat{f}(\xi-\eta)\hat{f}(\eta)d\eta\\
 &\quad+ \int_{\R^2}e^{is\Phi}m(\xi,\eta)\zeta\left(\frac{\abs{\xi-\eta}}{\abs{\eta}}\right)\hat{f}(\xi-\eta)\hat{f}(\eta)d\eta.
\end{aligned}
\end{equation*}
In the first of these terms we have $\abs{\eta}\lesssim \abs{\xi-\eta}$, hence $\abs{\xi}\lesssim\abs{\xi-\eta}$ as suggested. In contrast, for the second term $\abs{\xi-\eta}\lesssim\abs{\eta}$, leading us to change variables to $\sigma:=\xi-\eta$ in order to get
\begin{equation*}
 \int_{\R^2}e^{is\Phi}m(\xi,\xi-\sigma)\zeta\left(\frac{\abs{\sigma}}{\abs{\xi-\sigma}}\right)\hat{f}(\sigma)\hat{f}(\xi-\sigma)d\sigma,
\end{equation*}
since $\Phi(\xi,\eta)=\Phi(\xi,\xi-\eta)=\Phi(\xi,\sigma)$. Here we see that $\abs{\sigma}\lesssim\abs{\xi-\sigma}$ as claimed. 

\textbf{Conclusion:} In our considerations for the weighted estimate, through
\begin{equation}\label{eq:simpl}
 \int_{\R^2}e^{is\Phi}m(\xi,\eta)\hat{f}(\xi-\eta)\hat{f}(\eta)d\eta\sim \int_{\R^2}e^{is\Phi}\left(m+\bar{m}\right)(\xi,\eta)\hat{f}(\xi-\eta)\hat{f}(\eta) \chi_{\{\abs{\eta}\lesssim\abs{\xi-\eta}\}}d\eta
\end{equation}
we may reduce the Duhamel integral in \eqref{eq:DH} to two terms with null forms $m$ or $\bar{m}$ and localization $\abs{\eta}\lesssim\abs{\xi-\eta}$.

\emph{In the rest of the article we shall make use of this localization without writing it explicitly every time it appears. We note that in this setting $\bar{m}$ has a milder singularity than $m$, so often it will be enough to deal with $m$.}

\subsection{Estimate for \eqref{eq:dmult_term}}\label{sec:dmult_term}
Since there are no derivatives on the profiles and all multipliers factor (as described above), here it will suffice to invoke the energy estimate \eqref{eq:energy_ineq} and the dispersion. Up to a Riesz transform this gives
\begin{equation*}
 \begin{aligned}
  &\abs{ \int_0^t\int_{\R^2}e^{is\Phi}\abs{\xi}^2\partial_\xi m(\xi,\eta)\hat{f}(\xi-\eta)\hat{f}(\eta)d\eta ds}_{L^2_\xi}\\
  &\quad=\abs{\int_0^t\mathcal{F}^{-1}\left( \int_{\R^2}e^{is\Phi}\abs{\xi}^2\partial_\xi m(\xi,\eta)\hat{f}(\xi-\eta)\hat{f}(\eta)d\eta \right)ds}_{L^2_x}\\
  &\quad\lesssim \int_0^t\abs{u(s)}_{L^\infty}\abs{\omega(s)}_{H^2}ds,
 \end{aligned}
\end{equation*}
and similarly for the term with $\bar{m}$.

\subsection{Estimate for \eqref{eq:direct_weight}}\label{sec:direct_weight}
The idea here is that through \eqref{eq:simpl} in Section \ref{sec:simpl} we have arranged that the derivative $\partial_\xi$ falls on the high-frequency term, which allows $\abs{\xi}$ to be grouped with this derivative in order to close the estimate. The details are as follows.

Writing $\abs{\xi}^2=\left(\abs{\xi}^2-\abs{\xi-\eta}^2\right)+\abs{\xi-\eta}^2$ we first point out that by the divergence structure of the equation (Remark \ref{sec:div_structure}) the term with $m(\xi,\eta)\abs{\xi-\eta}^2 \partial_\xi\hat{f}(\xi-\eta)$ vanishes. Moreover, $\abs{\xi}^2-\abs{\xi-\eta}^2=2(\xi-\eta)\cdot\eta+\abs{\eta}^2$, hence 
\begin{equation*}
 \left(\abs{\xi}^2-\abs{\xi-\eta}^2\right)m(\xi,\eta)=2(\xi-\eta)\cdot\eta\frac{(\xi-\eta)\cdot\eta^\perp}{\abs{\eta}^2}+(\xi-\eta)\cdot\eta^\perp \lesssim_{CM}\abs{\xi-\eta}^2.
\end{equation*}
When dealing with $\bar{m}$ it suffices to notice that $\abs{\xi}\sim\abs{\xi-\eta}$.

Altogether we can thus bound two derivatives on \eqref{eq:direct_weight} by
\begin{equation*}
 \abs{\abs{\xi}^2\int_0^t\int_{\R^2}e^{is\Phi}m(\xi,\eta)\partial_\xi\hat{f}(\xi-\eta)\hat{f}(\eta)d\eta ds}\lesssim \int_0^t\abs{\omega(s)}_{L^\infty}\abs{\abs{\xi}^2\partial_\xi\hat{f}(s,\xi)}_{L_\xi^2} ds.
\end{equation*}

\subsection{High Frequency Estimates}\label{sec:highfreq}
For high frequencies we can make use of the energy estimates \eqref{eq:energy_ineq} to control the weight without resorting to any resonance analysis.

Recall that we denote by $P_{>N}$ the projection onto frequencies larger than $N$ in Fourier space (given by a smooth Fourier multiplier $\varphi_{>N}(\xi)$ with value 1 on frequencies larger than $N$ and supported on frequencies larger than $N/2$). The key point now is the following: By the Sobolev embedding we have for any $k\in\N$
\begin{equation}\label{eq:hfreqgain}
 \abs{P_{>N}f}_{L^\infty}\lesssim N^{-k} \abs{D^k P_{>N}f}_{L^\infty}\lesssim N^{-k}\abs{P_{>N}f}_{H^{3+k}}\leq N^{-k}\abs{f}_{H^{3+k}},
\end{equation}
so in particular for $l$ derivatives 
\begin{equation}\label{eq:hfreqgain2}
\begin{aligned}
 \abs{D^l f}_{L^\infty}&=\abs{P_{>N}D^l f+(1-P_{>N})D^l f}_{L^\infty}\\
 &\lesssim\abs{P_{>N}D^lf}_{L^\infty}+\abs{(1-P_{>N})D^lf}_{L^\infty}\\
 &\lesssim N^{-k}\abs{f}_{H^{3+k+l}}+N^l\abs{(1-P_{>N})f}_{L^\infty}\\
 &\lesssim N^{-k}\abs{f}_{H^{3+k+l}}+N^l\abs{f}_{L^\infty}.
\end{aligned}
\end{equation}
The moral is that by suitably splitting Fourier space we can assume the high frequencies to be small.

To estimate $\abs{\varphi_{>N}(\xi)\abs{\xi}^2\partial_\xi\hat{f}}_{L^2}$ we recall from Section \ref{sec:simpl} that by changing variables we may assume that $\abs{\xi-\eta}\sim\abs{\xi}> N$, so that it is enough to bound the term
\begin{equation*}
 \abs{\varphi_{>N}(\xi)\abs{\xi}^2\partial_\xi\int_0^t\int_{\R^2}e^{is\Phi}m(\xi,\eta)\widehat{P_{>N}f}(\xi-\eta)\hat{f}(\eta)d\eta ds}_{L^2_\xi}. 
\end{equation*}

We then need to control the same three types of terms as in \ref{sec:DHTerm}, now with an additional factor of $\varphi_{>N}$. As before, two of those are easily dealt with: When the weight $\partial_\xi$ falls on the null form $m$ or $\bar{m}$ we proceed as in \ref{sec:dmult_term}, directly invoking the energy estimates. In the worst case this yields
\begin{equation*}
 \begin{aligned}
  &\abs{ \varphi_{>N}(\xi)\int_0^t\int_{\R^2}e^{is\Phi}\abs{\xi}^2\partial_\xi m(\xi,\eta)\widehat{P_{>N}f}(\xi-\eta)\hat{f}(\eta)d\eta ds}_{L^2_\xi}\\
  &\quad\lesssim \int_0^t \abs{P_{>N}D^2 \omega(s)}_{L^2}\abs{u(s)}_{L^\infty}ds.
 \end{aligned}
\end{equation*}

Secondly, the case where the weight lands on the profile directly is dealt with as outlined in Section \ref{sec:direct_weight}.

On the other hand, term \eqref{eq:hard_piece} (when the weight hits the phase function) requires some elaboration: 
\subsubsection*{The Weight on the Phase Function}\label{sec:highfreqhard}
As mentioned above, for high frequencies we aim to control the time growth from the additional factor $s$ simply by energy estimates. A direct computation yields
\begin{equation*}
\begin{aligned}
 \partial_{\xi_1}\Phi&=\frac{\xi_2^2-\xi_1^2}{\abs{\xi}^4}-\frac{(\xi_2-\eta_2)^2-(\xi_1-\eta_1)^2}{\abs{\xi-\eta}^4},\\
 \partial_{\xi_2}\Phi&=\frac{2\xi_1\xi_2}{\abs{\xi}^4}-\frac{2(\xi_1-\eta_1)(\xi_2-\eta_2)}{\abs{\xi-\eta}^4},
\end{aligned}
\end{equation*}
so that the operator given by the symbol $\abs{\xi}^2\partial_\xi\Phi$ can be written as a sum of Riesz transforms in $x$ on the one hand, and of terms with derivatives of order $2$ in $\xi$ and order $-2$ in the variable $\xi-\eta$ on the other hand. Taking into account the null form as it factors onto the two profiles we thus obtain the bound 
\begin{equation*}
\begin{aligned}
 &\abs{\varphi_{>N}(\xi)\int_0^t\int_{\R^2}is\abs{\xi}^2\partial_\xi\Phi e^{is\Phi}m(\xi,\eta)\widehat{P_{>N}f}(\xi-\eta)\hat{f}(\eta)d\eta ds}_{L^2}\\
 &\quad\lesssim\int_0^t s \abs{P_{>N} D\omega}_{L^2}\abs{u}_{L^\infty}+ s \abs{D^2\left(e^{L_1t}(D^{-1}P_{>N}f)\; (e^{L_1t}D^{-1}f)\right)}_{L^2}\\
 &\hspace{3.5cm} +s \abs{D^2\left(e^{L_1t}(D^{-3}P_{>N}f)\; (e^{L_1t}Df)\right)}_{L^2} ds.
\end{aligned}
\end{equation*}

We now employ \eqref{eq:hfreqgain2} to deduce that for the worst case
\begin{equation*}
 \begin{aligned}
  \abs{D^2\left(e^{L_1t}(D^{-3}P_{>N}f)\; (e^{L_1t}Df)\right)}_{L^2}&\lesssim \abs{e^{L_1t}D^{-1} P_{>N}f}_{L^2}\abs{e^{L_1t}Df}_{L^\infty}\\
  &\quad+\abs{e^{L_1t}D^{-3}P_{>N}f}_{L^2}\abs{e^{L_1t}D^3f}_{L^\infty}\\
  &\hspace{-3.5cm}\lesssim \abs{P_{>N} u}_{L^2}\abs{D\omega}_{L^\infty}+\abs{D^{-3}P_{>N}\omega}_{L^\infty}\abs{D^3\omega}_{L^2}.
 \end{aligned}
\end{equation*}

\begin{prop}\label{prop:highfreq}
The high-frequency piece of the profile with weight and derivatives satisfies
  \begin{equation}\label{eq:2hiweights}
  \begin{aligned}
   \abs{P_{>N}D^2 xf(t)}_{L^2}&\lesssim \abs{P_{>N}D^2 xf_0}_{L^2}+ \int_0^t \abs{\omega(s)}_{L^\infty} \abs{D^2 xf(s)}_{L^2}ds\\
   &\hspace{-1cm} + 2^k \int_0^t \abs{\omega(s)}_{H^{3+k}}\left(\abs{\omega(s)}_{L^\infty}+\abs{u(s)}_{L^\infty}\right) \left(1+sN^{1-k}+sN^{-3-k}\right)ds\\
   &\hspace{-1cm} + 2^k\int_0^t sN^{-2k}\abs{\omega(s)}_{H^{4+k}}^2 ds.
  \end{aligned}
  \end{equation}
\end{prop}

\begin{proof}
 Note that also here we keep track of the dependence of our estimates on $k$ -- all implicit constants are independent of it.

 Adding up the various pieces of the estimates from sections \ref{sec:dmult_term}, \ref{sec:direct_weight} and the present \ref{sec:highfreq} we obtain
  \begin{equation}
  \begin{aligned}
   \abs{P_{>N}D^2 xf(t)}_{L^2}&\lesssim \abs{P_{>N}D^2 xf_0}_{L^2}+\int_0^t \abs{\omega(s)}_{L^\infty} \abs{D^2 xf(s)}_{L^2}ds\\
    &\quad+ \int_0^t \abs{u(s)}_{L^\infty}\abs{\omega(s)}_{H^2}ds\\
    &\quad+ \int_0^t s\abs{P_{>N}u(s)}_{L^2}\abs{D\omega(s)}_{L^\infty}ds\\
    &\quad+ \int_0^t s\abs{D^{-3}P_{>N}\omega(s)}_{L^\infty}\abs{D^3\omega(s)}_{L^2}ds.
  \end{aligned}
  \end{equation}
 Now we invoke the splitting into high and low frequencies as in \eqref{eq:hfreqgain2} (with $l=1$) when we have a derivative on $\omega$ in $L^\infty$, and in the same spirit we use that $\abs{P_{>N}f}_{L^2}\lesssim 2^k N^{-k}\abs{f}_{H^k}$. This yields
  \begin{equation*}
  \begin{aligned}
   \abs{P_{>N}D^2 xf(t)}_{L^2}&\lesssim \abs{P_{>N}D^2 xf_0}_{L^2} + \int_0^t \abs{\omega(s)}_{L^\infty} \abs{D^2 xf(s)}_{L^2}ds\\
    &\quad+ \int_0^t \abs{u(s)}_{L^\infty}\abs{\omega(s)}_{H^2}ds\\
    &\quad+ 2^k\int_0^t sN^{-k}\abs{\omega(s)}_{H^{k}}\left(N^{-k}\abs{\omega(s)}_{H^{4+k}}+N\abs{\omega(s)}_{L^\infty}\right)ds\\
    &\quad+ 2^k\int_0^t sN^{-3}\abs{\omega(s)}_{L^\infty}N^{-k}\abs{\omega(s)}_{H^{3+k}}ds,
  \end{aligned}
  \end{equation*}
 from which \eqref{eq:2hiweights} follows.
\end{proof}

\subsection{Reduction to Low Frequencies}\label{sec:redlowfreq}
In the previous section we used that if $\xi$ is large, then either $\xi-\eta$ or $\eta$ has to be large as well, so we can gain decay through Sobolev embedding. If on the other hand now $\xi$ is small but $\eta$ large, the same estimates go through.

\textbf{Conclusion:} In what follows we may assume that all frequencies are bounded by some $N>0$, the time-dependence of which will be made explicit later. This means that we can reduce to a Duhamel formula of the type
 \begin{equation*}
  \widehat{P_{\leq N}f}(t,\xi)=\widehat{P_{\leq N}f_0}(\xi)+\int_0^t\int_{\R^2}e^{is\Phi}m(\xi,\eta)\widehat{P_{\leq N}f}(s,\xi-\eta)\widehat{P_{\leq N}f}(s,\eta)d\eta ds.
 \end{equation*}

For such a term we will then be able to show the following proposition:
\begin{prop}\label{prop:lowfreq}
 The low-frequency piece of the profile with weight and two derivatives satisfies
   \begin{equation}\label{eq:2weights}
  \begin{aligned}
    \abs{P_{\leq N}D^2xf(t)}_{L^2}&\lesssim \abs{D^2 xf_0}_{L^2}+\int_0^t \abs{D^2xf(s)}_{L^2}\abs{\omega(s)}_{L^\infty}ds\\
    &\quad +\int_0^t \left(\abs{\omega(s)}_{L^\infty}+\abs{u(s)}_{L^\infty}\right)\abs{\omega(s)}_{H^1}ds\\
    &\quad +\int_0^t s \left(\abs{u(s)}^2_{L^\infty}+\abs{u(s)}_{L^\infty}\abs{Du(s)}_{L^\infty}\right)\abs{\omega(s)}_{H^2}ds\\
    &\quad +\int_0^t \rho^{-\frac{3}{2}}N^{\rho}\abs{D^2xf(s)}_{L^2}\abs{\omega(s)}_{L^{\frac{2}{\rho}}}ds\\
    &\quad +\int_0^t \rho^{-\frac{3}{2}}N^{\rho}\left(\abs{\omega(s)}_{L^\frac{2}{\rho}}+\abs{u(s)}_{L^\frac{2}{\rho}}\right)\abs{\omega(s)}_{H^1}ds,
  \end{aligned}
 \end{equation}
 where all frequencies on the right-hand side may also be assumed to be bounded by $N>0$ and $0<\rho\ll 1$ (to be chosen later).
\end{prop}

Its proof involves the resonance analysis hinted at earlier and is carried out next (see Section \ref{sec:splitting}).

\section{Weighted Estimate -- Resonance Analysis}\label{sec:res}
In this chapter we complete the proof of the weighted estimates, following the strategy described in the outline of the proofs of Propositions \ref{prop:added_derivative}-\ref{prop:lowfreq} at the beginning of Section \ref{sec:weights_red}. In Sections \ref{sec:direct_approach} and \ref{sec:time_res} we first give a basic overview of useful techniques in the context of the resonance analysis for the weighted estimate. Section \ref{sec:splitting} then lays out the relevant details in the case of two derivatives with a weight, thus concluding the proof of Proposition \ref{prop:lowfreq}. Finally, in Section \ref{sec:add_deriv} we will see how one can obtain bounds for an additional third derivative on the weighted profile function. This finishes the proof of Proposition \ref{prop:added_derivative}.

Many of the estimates in this section rely on the computations in Appendix \ref{apdx:comp}.

In view of the analysis in the previous chapter, in the case of two derivatives on a weighted profile function we are left with estimating in $L^2$ the time integral
\begin{equation*}
 \abs{\xi}^2\int_0^t\int_{\R^2}is\partial_\xi\Phi e^{is\Phi}m(\xi,\eta)\hat{f}(\xi-\eta)\hat{f}(\eta)d\eta ds.
\end{equation*}
We recall that we may assume all frequencies to be bounded by some $N>0$ and that $\abs{\eta}\lesssim\abs{\xi-\eta}$. For simplicity of notation we will not explicitly express the localizations, but they will be assumed implicitly in all the relevant terms that follow.

\subsection{The ``Direct'' Approach}\label{sec:direct_approach}
A natural thing to do is to try to get rid of the extra time factor $s$ through space resonances: Using
\begin{equation}\label{eq:ibpid}
 e^{is\Phi}=\frac{1}{is\abs{\nabla_\eta\Phi}^2}\nabla_\eta\Phi\cdot\nabla_\eta e^{is\Phi}
\end{equation}
to integrate by parts in $\eta$ leads to terms of the type
\begin{equation*}
 \frac{\partial_{\xi_j}\Phi}{\abs{\nabla_\eta\Phi}^2}\partial_{\eta_l}\Phi \;m(\xi,\eta)
\end{equation*}
and first order derivatives in $\eta_l$ thereof, $j,l=1,2$.

However, one computes that
\begin{equation}\label{eq:phasequot}
 \frac{\abs{\nabla_\xi\Phi}}{\abs{\nabla_\eta\Phi}}= \frac{\abs{\eta-2\xi} \abs{\eta}^3}{\abs{\xi-2\eta}\abs{\xi}^3},
\end{equation}
so these quotients have singularities at $\xi=0$ and at $\xi=2\eta$ and will only be useful in certain subdomains of Fourier space. On the other hand we point out that the null structure of $m$ helps with some cancellation: since $\xi^\perp\ccdot\eta=(\xi-\eta)^\perp\ccdot\eta=(\xi-2\eta)^\perp\ccdot\eta$, one order of singularity can be dealt with in some cases. However, notice that an integration by parts will yield a term without null structure (since the derivative may fall on $m$), in which case the singularity of $\xi-2\eta$ needs to be dealt with by other means. This highlights the importance of a suitable choice of region on which to use this approach.

In more detail, an integration by parts in $\eta$ using \eqref{eq:ibpid}
leaves us with two types of terms:
\begin{enumerate}
 \item\label{it:t1} When the derivative in $\eta$ lands on one of the profiles: $$\int_0^t\int_{\R^2}e^{is\Phi}\abs{\xi}^2 m(\xi,\eta)\frac{\partial_{\xi_j}\Phi}{\abs{\nabla_\eta\Phi}^2}\nabla_\eta\Phi \cdot \nabla_\eta\hat{f}(\xi-\eta)\hat{f}(\eta)d\eta ds$$ or $$\int_0^t\int_{\R^2}e^{is\Phi}\abs{\xi}^2m(\xi,\eta)\frac{\partial_{\xi_j}\Phi}{\abs{\nabla_\eta\Phi}^2}\nabla_\eta\Phi\cdot \hat{f}(\xi-\eta)\nabla_\eta\hat{f}(\eta)d\eta ds.$$ This is the most delicate case to deal with, since it involves the weighted estimate itself -- in order to be able to close our estimates we need to make sure that the weights on the profile come with at least\footnote{Since all frequencies are assumed to be bounded by $N>0$, extra derivatives just amount to a loss of some power of $N$.} two derivatives. On the positive side we remark that, as noted above, the null structure can help in controlling the singularities $\xi=2\eta$ and $\xi=0$ in this case.
 
 In total, here we have a multiplier with $$\frac{\abs{\nabla_\xi\Phi}}{\abs{\nabla_\eta\Phi}}\abs{m}\leq \frac{\abs{\eta-2\xi} \abs{\eta}^2}{\abs{\xi-2\eta}\abs{\xi}^3}\min\{\abs{\xi}, \abs{\xi-2\eta}\},$$ so the singularity in $\xi$ is only of order $2$ and is thus fully compensated by $\abs{\xi}^2$.
 
 \item\label{it:t2} When the derivative in $\eta$ lands on the Fourier multipliers: $$\int_0^t\int_{\R^2}e^{is\Phi}\abs{\xi}^2\nabla_\eta\ccdot\left(\frac{\partial_{\xi_j}\Phi}{\abs{\nabla_\eta\Phi}^2}\nabla_\eta\Phi\; m(\xi,\eta)\right)\hat{f}(\xi-\eta)\hat{f}(\eta)d\eta ds.$$ Unlike the other terms, this one can be dealt with purely by energy estimates. For this term we cannot make use of the null structure to deal with the singularities arising, since it will be lost once the derivative in $\eta$ falls on it.
 
 Clearly we have $$\nabla_\eta\ccdot\left(\frac{\partial_{\xi_j}\Phi}{\abs{\nabla_\eta\Phi}^2}\nabla_\eta\Phi\; m\right)=\nabla_\eta\ccdot\left(\nabla_\eta\Phi\frac{\partial_{\xi_j}\Phi}{\abs{\nabla_\eta\Phi}^2}\right)m+\left(\nabla_\eta\Phi\frac{\partial_{\xi_j}\Phi}{\abs{\nabla_\eta\Phi}^2}\right)\ccdot \nabla_\eta m.$$

 For the first term we point out that
 \begin{equation*}
  \partial_{\eta_1}^2\Phi+\partial_{\eta_2}^2\Phi=0,
 \end{equation*}
 so that when the integration by parts produces a derivative on the phase quotients we have
 \begin{equation}\label{eq:Dphasequot}
 \begin{aligned}
  \nabla_\eta\ccdot\left(\nabla_\eta\Phi\frac{\partial_{\xi_j}\Phi}{\abs{\nabla_\eta\Phi}^2}\right)&=\nabla_\eta\Phi\ccdot\nabla_\eta\left(\frac{\partial_{\xi_j}\Phi}{\abs{\nabla_\eta\Phi}^2}\right)\\
   &\hspace{-2cm}=\frac{1}{2\abs{\nabla_\eta\Phi}^2}\partial_{\xi_j}\left(\abs{\nabla_\eta\Phi}^2\right)-\frac{1}{\abs{\nabla_\eta\Phi}^4}\partial_{\xi_j}\Phi\nabla_\eta\Phi \ccdot\nabla_\eta\left(\abs{\nabla_\eta\Phi}^2\right).
 \end{aligned}
 \end{equation}
 From the computations in Appendix \ref{apdx:deriv} we see that these terms are of the order $$ \abs{\frac{1}{2\abs{\nabla_\eta\Phi}^2}\partial_{\xi_j}\left(\abs{\nabla_\eta\Phi}^2\right) }\lesssim\frac{\abs{\eta}^2}{\abs{\xi-2\eta}\abs{\xi}\abs{\xi-\eta}}$$ and
 \begin{equation*}
 \begin{aligned}
  \abs{\frac{1}{\abs{\nabla_\eta\Phi}^4}\partial_{\xi_j}\Phi\nabla_\eta\Phi \ccdot\nabla_\eta\left(\abs{\nabla_\eta\Phi}^2\right)}&\lesssim \frac{\abs{\xi-\eta}^4\abs{\eta}^4}{\abs{\xi-2\eta}^2\abs{\xi}^2}\left(\abs{\xi-\eta}^{-2}+\abs{\xi}^{-2}\right)\\
  &\qquad\times\left(\abs{\eta}^{-3}+\abs{\xi-\eta}^{-3}\right)
 \end{aligned}
 \end{equation*}

 In the computations later on we will combine this with the estimates $\abs{m}\leq \min\{\abs{\xi},\abs{\xi-2\eta}\}\abs{\eta}^{-1}$ and $\abs{\nabla_\eta m}\leq \abs{\xi}\abs{\eta}^{-2}$ (see also Appendix \ref{apdx:comp}).
\end{enumerate}

\subsection{A Time Resonance Estimate}\label{sec:time_res}
As an illustration for the estimates needed later on we treat here the case of pure time resonances (which will come up later on for Region 1 in Section \ref{sec:splitting}), i.e.\ we demonstrate how to estimate a term of the form
\begin{equation*}
 \int_0^t\int_{\R^2}is\Phi e^{is\Phi}m(\xi,\eta)\hat{f}(\xi-\eta)\hat{f}(\eta)d\eta ds.
\end{equation*}

Clearly $is\Phi e^{is\Phi}=s\partial_s\left(e^{is\Phi}\right)$, so that integration by parts in time yields the following pieces:
\begin{equation*}
\begin{aligned}
 \int_0^t\int_{\R^2}&is\Phi e^{is\Phi}m(\xi,\eta)\hat{f}(s,\xi-\eta)\hat{f}(s,\eta)d\eta ds\\
 &\quad =it\int_{\R^2}e^{it\Phi}m(\xi,\eta)\hat{f}(t,\xi-\eta)\hat{f}(t,\eta)d\eta\\
 &\qquad -\int_0^t\int_{\R^2} e^{is\Phi}m(\xi,\eta)\hat{f}(s,\xi-\eta)\hat{f}(s,\eta)d\eta ds\\
 &\qquad -\int_0^t\int_{\R^2} s e^{is\Phi}m(\xi,\eta)\partial_s\hat{f}(s,\xi-\eta)\hat{f}(s,\eta)d\eta ds\\
 &\qquad -\int_0^t\int_{\R^2} s e^{is\Phi}m(\xi,\eta)\hat{f}(s,\xi-\eta)\partial_s\hat{f}(s,\eta)d\eta ds.
\end{aligned}
\end{equation*}

We recall that $u$ decays in $L^\infty$ at the same rate as $\omega$ and $Du$ (see also Remark \ref{rem:u_Du_decay}). In out bootstrapping scheme all these terms will then be bounded, as the $L^\infty$ decay gives integrability on arbitrarily long time intervals. Note that the terms with an extra factor of $s$ inside the time integral are actually cubic, since by the equation we can substitute a bilinear term for $\partial_s\hat{f}$. In these cases we can take advantage of the decay twice. We proceed to give the relevant details.

\subsubsection*{The first term}
We note that
\begin{equation*}
 \mathcal{F}^{-1}\left(\int_{\R^2}e^{it\Phi}m(\xi,\eta)\hat{f}(t,\xi-\eta)\hat{f}(t,\eta)d\eta\right)=u\ccdot\nabla\omega(t),
\end{equation*}
so this is easily controlled:
\begin{equation*}
 \abs{t\int_{\R^2}e^{it\Phi}m(\xi,\eta)\hat{f}(t,\xi-\eta)\hat{f}(t,\eta)d\eta}_{L^2_\xi}=\abs{t u\ccdot\nabla\omega(t)}_{L^2_x}\lesssim t \abs{u(t)}_{L^\infty}\abs{\nabla\omega(t)}_{L^2}.
\end{equation*}

\subsubsection*{The second term}
This is of the same order as the first term, since
\begin{equation*}
 \int_0^t\abs{\int_{\R^2} e^{is\Phi}m(\xi,\eta)\hat{f}(s,\xi-\eta)\hat{f}(s,\eta)d\eta }_{L^2_\xi}ds=\int_0^t\abs{u\ccdot\nabla\omega(s)}_{L^2_x}ds.
\end{equation*}

\subsubsection*{The third and fourth terms}
The third and fourth terms can be estimated essentially in the same way, so we only give the details for the third one. From the definition of the profile $f(t):=e^{-L_1t}\omega(t)$ it follows that
\begin{equation*}
 \partial_t f(t)=e^{-L_1t}(-L_1\omega(t)+\partial_t\omega(t))=e^{-L_1t}(u\ccdot\nabla\omega(t)).
\end{equation*}
Hence the standard estimate gives
\begin{equation*}
 \begin{aligned}
  &\int_0^t\abs{\int_{\R^2} s e^{is\Phi}m(\xi,\eta)\partial_s \hat{f}(s,\xi-\eta)\hat{f}(s,\eta)d\eta}_{L^2_\xi} ds\\
  &\quad =\int_0^t s\abs{u\ccdot\nabla\left(e^{L_1t}(u\ccdot\nabla\omega)\right)}_{L^2_x} ds\\
  &\quad \leq \int_0^t s\abs{u(s)}_{L^\infty}\abs{\nabla\left(u\ccdot\nabla\omega\right)}_{L^2}ds\\
  &\quad \leq \int_0^t s\abs{u(s)}_{L^\infty}\left(\abs{u(s)}_{L^\infty}\abs{\omega(s)}_{H^2}+\abs{Du(s)}_{L^\infty}\abs{\omega(s)}_{H^1}\right)ds.
 \end{aligned}
\end{equation*}

\subsection{The Splitting of Fourier Space for Low Frequencies}\label{sec:splitting}
We will use the following divisions of Fourier space into regions where various techniques will be employed to control \eqref{eq:hard_piece}, the most delicate term that comes up when we estimate a weight on the profile. We recall (Section \ref{sec:redlowfreq}) that for the estimates themselves we may assume that all frequencies are bounded by some $N>0$ (which of course is not relevant for the splitting, but rather for the potential derivative loss), as well as the restriction that $\abs{\eta}\lesssim\abs{\xi-\eta}$, as discussed in Section \ref{sec:simpl}.

We point out our use of the following
\begin{notation}
 In this section, all integrals in $\eta$ have a localization (which we may not write explicitly) given by the particular subsection they are found in.
\end{notation}

\begin{description}
 \item[\textbf{Region 1}]
  \begin{equation}\label{eq:G1.1} \frac{\abs{\eta}}{100} \leq \abs{\xi} \leq 100\abs{\eta} \end{equation}
  and
  \begin{equation}\label{eq:G1.2}  \frac{\abs{\eta}}{10000} \leq \abs{\xi-\eta} \leq 10000\abs{\eta} \end{equation}

 \item[\textbf{Region 2}]
  \begin{equation}\label{eq:G3} \abs{\xi} \leq \frac{\abs{\eta}}{100} \end{equation}

 \item[\textbf{Region 3}]
  \begin{equation}\label{eq:G4} \abs{\xi} \geq {100\abs{\eta}} \end{equation}
\end{description}

Let us now discuss how to obtain the relevant bounds for \eqref{eq:hard_piece}. As outlined further above, this will involve a study of the resonances (by which the above splitting of Fourier space is inspired) in the various domains and subsequent use of the techniques discussed in Sections \ref{sec:direct_approach} and \ref{sec:time_res}.

The result of this process is the estimate \eqref{eq:2weights} from Proposition \ref{prop:lowfreq} ($0<\rho\ll 1)$:
 \begin{equation*}
  \begin{aligned}
    \abs{D^2xf(t)}_{L^2}&\lesssim \abs{D^2xf_0}_{L^2}+\int_0^t \abs{D^2xf(s)}_{L^2}\abs{\omega(s)}_{L^\infty}ds\\
    &\quad+\int_0^t \left(\abs{\omega(s)}_{L^\infty}+\abs{u(s)}_{L^\infty}\right)\abs{\omega(s)}_{H^1}ds\\
    &\quad +\int_0^t s \left(\abs{u(s)}^2_{L^\infty}+\abs{u(s)}_{L^\infty}\abs{Du(s)}_{L^\infty}\right)\abs{\omega(s)}_{H^2}ds\\
    &\quad +\int_0^t \rho^{-\frac{3}{2}}N^{\rho}\abs{D^2xf(s)}_{L^2}\abs{\omega(s)}_{L^{\frac{2}{\rho}}}ds\\
    &\quad +\int_0^t \rho^{-\frac{3}{2}}N^{\rho}\left(\abs{\omega(s)}_{L^\frac{2}{\rho}}+\abs{u(s)}_{L^\frac{2}{\rho}}\right)\abs{\omega(s)}_{H^1}ds.
  \end{aligned}
 \end{equation*}
Here the first three integrals come from the estimates in Regions 1 and 2 via the Coifman-Meyer Theorem \ref{thm:CM}, whereas the last two integral terms are the product of an estimate using Theorem \ref{thm:multhm} in Region 3. The basic techniques for the proof of this estimate have been introduced above, so now we complete the steps outlined there by discussing the bounds for the various regions of our splitting.

\subsubsection{Region 1 -- \eqref{eq:G1.1} and \eqref{eq:G1.2}: $\abs{\xi}\sim\abs{\eta}\sim\abs{\xi-\eta}$}
Here all frequencies are of comparable size and hence the multipliers can be controlled via Coifman-Meyer type estimates. The main difficulty is to deal with the space resonance $\xi=2\eta$. We thus split into two cases depending on the size of $\abs{\xi-2\eta}$: Away from the space resonance a direct integration by parts goes through, whereas for ``small'' $\abs{\xi-2\eta}$ we further (anisotropically) split into two subcases and use time resonances or a ``good'' direction in which to integrate by parts. These steps rely on the null structure of the bilinearity.

\subsubsection*{Case 1:} \begin{equation}\label{eq:G1C1} \abs{\xi-2\eta} \geq \frac{\abs{\eta}}{1000}\end{equation}
In this case all singularities are avoided and we will only need to integrate by parts in $\eta$. Indeed, recall from \eqref{eq:phasequot} in Section \ref{sec:direct_approach} that 
\begin{equation*}
\frac{\abs{\nabla_\xi\Phi}}{\abs{\nabla_\eta\Phi}}= \frac{\abs{\eta-2\xi} \abs{\eta}^3}{\abs{\xi-2\eta}\abs{\xi}^3}.
\end{equation*}
Under \eqref{eq:G1.1}, \eqref{eq:G1.2} and \eqref{eq:G1C1} we see directly that\footnote{Recall the notation $\lesssim_{CM}$ for boundedness in the sense of a Coifman-Meyer multiplier from page \pageref{rmk:CM_not}.}
$$\frac{\abs{\nabla_\xi\Phi}}{\abs{\nabla_\eta\Phi}} \lesssim_{CM} 1.$$
Hence this term (of type \eqref{it:t1} as described in Section \ref{sec:direct_approach}) gives back the weight $D^2xf$ in $L^2$ and $\omega$ in $L^\infty$, i.e.\ for it we have the bound
$$\int_0^t \abs{\abs{\xi}^2\partial_\xi\hat{f}(s)}_{L^2} \abs{\omega(s)}_{L^\infty} ds.$$

Similarly, from \eqref{eq:Dphasequot} and the observations in Section \ref{sec:direct_approach} we also deduce that
$$\abs{\nabla_\eta\cdot\left(\nabla_\eta\Phi\frac{\partial_{\xi_i}\Phi}{\abs{\nabla_\eta\Phi}^2}\right) }\lesssim_{CM} \abs{\eta}^{-1},$$
so that the multipliers for terms of type \eqref{it:t2} are of order $\abs{\xi}^2\abs{\eta}^{-1}$ (in the sense of Theorem \ref{thm:CM}) and we can bound them by
$$\int_0^t \abs{\omega(s)}_{H^2}\abs{u(s)}_{L^\infty} ds.$$

\subsubsection*{Case 2:} \begin{equation}\label{eq:G1C2} \abs{\xi-2\eta} \leq \frac{\abs{\eta}}{1000}\end{equation}
Here we observe that
\begin{equation}\label{eq:G1C2.1}
 \abs{\Phi} \geq \frac{\frac{6}{10}\abs{\xi_1}-\frac{2}{1000}\abs{\eta_1}}{\abs{\eta}^2}.
\end{equation}
Indeed, this is gotten simply by noting that $$\Phi= \xi_1 \left( \frac{1}{\abs{\xi}^2} - \frac{1}{\abs{\xi-\eta}^2} \right) + \eta_1 \left( \frac{1}{\abs{\xi-\eta}^2} - \frac{1}{\abs{\eta}^2} \right)$$
and invoking \eqref{eq:G1C2} and \eqref{eq:G1.1} and \eqref{eq:G1.2}.  

\begin{proof}[Proof of \eqref{eq:G1C2.1}]
We begin by observing that $\abs{\xi-2\eta}\leq \frac{\abs{\eta}}{1000}$ implies that $\frac{2001}{1000}\abs{\eta}\geq \abs{\xi}\geq \frac{1999}{1000}\abs{\eta}$ and thus $\frac{1001}{1000}\abs{\eta}\geq \abs{\xi-\eta}\geq \frac{999}{1000}\abs{\eta}$.

Now one has $\frac{1}{\abs{\xi-\eta}^2}-\frac{1}{\abs{\eta}^2}=\left(\frac{1}{\abs{\xi-\eta}}-\frac{1}{\abs{\eta}}\right)\left(\frac{1}{\abs{\xi-\eta}}+\frac{1}{\abs{\eta}}\right)$ and thus
$$\abs{\frac{1}{\abs{\xi-\eta}}-\frac{1}{\abs{\eta}}}=\frac{\abs{\abs{\xi-\eta}-\abs{\eta}}}{\abs{\eta}\abs{\xi-\eta}}\leq \frac{\abs{\xi-2\eta}}{\abs{\eta}\abs{\xi-\eta}}\leq \frac{1}{1000\abs{\xi-\eta}},$$
which yields
$$\abs{\frac{1}{\abs{\xi-\eta}^2}-\frac{1}{\abs{\eta}^2}}\leq \frac{1}{1000\abs{\eta}\abs{\xi-\eta}}\leq \frac{1}{999\abs{\eta}^2}.$$

On the other hand,
\begin{equation*}
\begin{aligned}
 \abs{\frac{1}{\abs{\xi}^2}-\frac{1}{\abs{\xi-\eta}^2}}&=\abs{\frac{1}{\abs{\xi}}-\frac{1}{\abs{\xi-\eta}}} \abs{\frac{1}{\abs{\xi}}+\frac{1}{\abs{\xi-\eta}}}=\abs{\frac{\abs{\xi}-\abs{\xi-\eta}}{\abs{\xi}\abs{\xi-\eta}}} \abs{\frac{1}{\abs{\xi}}+\frac{1}{\abs{\xi-\eta}}}\\ 
  &\geq \abs{\frac{\abs{\eta}}{\frac{1999}{1000}\abs{\eta}\frac{999}{1000}\abs{\eta}}} \abs{\frac{1000}{2001}\frac{1}{\abs{\eta}}+ \frac{1000}{1001}\frac{1}{\abs{\eta}}}\\
  &\geq \frac{6}{10}\frac{1}{\abs{\eta}^2}.
\end{aligned}
\end{equation*}
Altogether we have thus established \eqref{eq:G1C2.1}.
\end{proof}

This leads us to two subcases. 

\paragraph{\underline{Subcase A}:}  \begin{equation}\label{eq:G1C2A} \abs{\xi_1}\geq \frac{\abs{\eta_1}}{100}\end{equation}

In this case, from \eqref{eq:G1C2.1} we deduce that
$$\abs{\Phi} \geq \frac{\abs{\xi_1}}{2\abs{\eta}^2}.$$
But then 
$$\frac{\abs{m}}{\abs{\Phi}}=\frac{1}{\abs{\Phi}}\frac{\abs{\xi_1\eta_2-\xi_2\eta_1}}{\abs{\eta}^2}\leq \frac{2\abs{\eta}^2}{\abs{\xi_1}}\frac{\abs{\xi_1} (\abs{\eta}+\abs{\xi})}{\abs{\eta}^2} \lesssim\abs{\eta}$$
by virtue of \eqref{eq:G1.1}, \eqref{eq:G1.2} and \eqref{eq:G1C2A}. It follows that $\frac{m}{\Phi}\lesssim_{CM}\abs{\eta}$ and hence we can integrate by parts in time to obtain estimates similar to the ones before (where the shorthand $\chi_A$ denotes the localizations of this subcase):
\begin{equation}
 \begin{aligned}
  &\int_0^t\int_{\R^2}is\abs{\xi}^2\partial_\xi\Phi e^{is\Phi}m(\xi,\eta)\chi_{A}\,\hat{f}(\xi-\eta)\hat{f}(\eta)d\eta ds\\
  &\quad =\int_0^t\int_{\R^2}is\abs{\xi}^2\partial_\xi\Phi \frac{1}{i\Phi}\partial_s\left(e^{is\Phi}\right)m(\xi,\eta)\chi_{A}\,\hat{f}(\xi-\eta)\hat{f}(\eta)d\eta ds\\
  &\quad =it\int_{\R^2}e^{it\Phi}\abs{\xi}^2\partial_\xi\Phi\frac{m}{\Phi}\chi_{A}\,\hat{f}(\xi-\eta)\hat{f}(\eta)d\eta\\
  &\qquad -\int_0^t\int_{\R^2}e^{is\Phi}\abs{\xi}^2\partial_\xi\Phi\frac{m}{\Phi}\chi_{A}\,\hat{f}(\xi-\eta)\hat{f}(\eta)d\eta ds\\
  &\qquad -\int_0^t\int_{\R^2}se^{is\Phi}\abs{\xi}^2\partial_\xi\Phi\frac{m}{\Phi}\chi_{A}\,\partial_s\hat{f}(\xi-\eta)\hat{f}(\eta)d\eta ds\\
  &\qquad -\int_0^t\int_{\R^2}se^{is\Phi}\abs{\xi}^2\partial_\xi\Phi\frac{m}{\Phi}\chi_{A}\,\hat{f}(\xi-\eta)\partial_s\hat{f}(\eta)d\eta ds.\\
 \end{aligned}
\end{equation}

These terms can now essentially be estimated as outlined in Section \ref{sec:time_res} (note that the additional multipliers $\abs{\xi}^2\partial_\xi\Phi$ are of order one, since $\abs{\xi}\sim\abs{\xi-\eta}$ in Region 1). In particular, no weights are required and there is no derivative loss in $L^\infty$.

\paragraph{\underline{Subcase B}:}  \begin{equation}\label{eq:G1C2B} \abs{\xi_1}< \frac{\abs{\eta_1}}{100}\end{equation}

In this case, integration by parts in time will not avail us. However, due to the constraints on $\xi_1$ and the other variables we may integrate by parts in $\eta_2$ to obtain something useful:
\begin{lemma}\label{lem:C2SC2}
 Assume that $\abs{\xi-2\eta} \leq \frac{\abs{\eta}}{1000}$ and $\abs{\xi_1}< \frac{\abs{\eta_1}}{100}$. Then
 \begin{equation}
  \frac{m}{\partial_{\eta_2} \Phi}\lesssim_{CM} \abs{\eta}^2
 \end{equation}
 and 
 \begin{equation}\label{eq:Deta2phasemult}
 \partial_{\eta_2} \left(\frac{m}{\partial_{\eta_2} \Phi}\right) \lesssim_{CM} \abs{\eta}.
\end{equation}
\end{lemma}

These estimates allow us to integrate by parts in $\eta_2$ in this region. One notes that as necessary we end up with two derivatives on the weight and lose one derivative in the energy estimate -- the arguments are completely analogous to our prior treatment of Case 1 for Region 1 (this is \eqref{eq:G1.1} and \eqref{eq:G1.2} with \eqref{eq:G1C1}): the integration by parts in space proceeds in close analogy to the corresponding observations on page \pageref{eq:G1C1} and we end up with the bound
$$\int_0^t \abs{D^2xf(s)}_{L^2}\abs{\omega(s)}_{L^\infty} ds+\int_0^t \abs{\omega(s)}_{H^1}\abs{\omega(s)}_{L^\infty} ds. $$
in this subcase.

\begin{proof}[Proof of Lemma \ref{lem:C2SC2}]
We begin by recalling that
\begin{equation}\label{eq:Deta2Phi}
\partial_{\eta_2} \Phi = -2{\eta_1\eta_2}\left(\frac{1}{\abs{\eta}^4}- \frac{1}{\abs{\xi-\eta}^4} \right) +2\frac{\eta_1\xi_2+\xi_1\eta_2-\xi_1\xi_2}{\abs{\xi-\eta}^4}.
\end{equation}




Note also that
$$\abs{\xi}\geq 2\abs{\eta}-\abs{\xi-2\eta}\geq \left(2-\frac{1}{1000}\right)\abs{\eta},$$
which implies
$$\abs{\xi_2}\geq \abs{\xi}-\abs{\xi_1}\geq \left(2-\frac{1}{1000}-\frac{1}{100}\right)\abs{\eta}\geq \frac{3}{2}\abs{\eta}.$$

In direct analogy we have
$$\abs{\xi-\eta}\geq\abs{\eta}-\abs{\xi-2\eta}\geq \left(1-\frac{1}{1000}\right)\abs{\eta}$$
and
$$\abs{\xi-\eta}\geq\abs{\eta}+\abs{\xi-2\eta}\geq \left(1+\frac{1}{1000}\right)\abs{\eta}$$
and thus
$$\abs{1- \frac{\abs{\xi-\eta}^4}{\abs{\eta}^4}} < \frac{1}{100}.$$

To estimate $\partial_{\eta_2}\Phi$ we use that for the terms in \eqref{eq:Deta2Phi} we have $\abs{\eta_1\xi_2}\geq \abs{\eta_1}\abs{\eta}$ and $\abs{\xi_1(\eta_2-\xi_2)}\leq \frac{\abs{\eta_1}}{100}\abs{\eta_2-\xi_2}\leq \frac{4}{100}\abs{\eta_1}\abs{\eta}$, so that
$$\abs{\eta_1\xi_2+\xi_1\eta_2-\xi_1\xi_2}\geq\frac{1}{4}\abs{\eta_1}\abs{\eta}.$$
In addition, 
$$\abs{\eta_1\eta_2 \left(1- \frac{\abs{\xi-\eta}^4}{\abs{\eta}^4}\right)}\leq \frac{1}{100}\abs{\eta_1}\abs{\eta},$$
which gives the claimed bound
$$\abs{\partial_{\eta_2} \Phi}\geq 2\left(\frac{1}{4}-\frac{1}{100}\right)\frac{\abs{\eta_1}\abs{\eta}}{\abs{\xi-\eta}}\geq \frac{\abs{\eta_1} \abs{\eta}}{4\abs{\xi-\eta}^4}.$$

For the null form we notice that with the present restrictions on $\xi, \eta$ we have 
$$\abs{\xi\cdot\eta^\perp}\sim \abs{\eta_1}\abs{\eta}:$$ 
On the one hand $\abs{\xi\cdot\eta^\perp}=\abs{-\xi_1\eta_2+\xi_2\eta_1} \geq \frac{1}{2}\abs{\eta_1}\abs{\eta}$, since $\abs{\xi_1\eta_2}\leq \frac{\abs{\eta_1}}{100}\abs{\eta}$ and $\abs{\xi_2\eta_1}\geq \abs{\eta}\abs{\eta_1}$. On the other, $\abs{\xi\cdot\eta^\perp}\leq \abs{\xi_1\eta_2}+\abs{\xi_2\eta_1}\leq \abs{\eta_1}(\frac{\abs{\eta_2}}{100}+\abs{\xi_2})\leq 4\abs{\eta_1}\abs{\eta}$.

To control the remaining term \eqref{eq:Deta2phasemult} for the integration by parts we write
$$\partial_{\eta_2} \left(\frac{m}{\partial_{\eta_2} \Phi}\right)=\frac{1}{\abs{\partial_{\eta_2} \Phi}^2}\left(\partial_{\eta_2}m\;\partial_{\eta_2}\Phi - m\;\partial^2_{\eta_2}\Phi\right).$$

We recall from Appendix \ref{apdx:deriv} that
$$\partial^2_{\eta_2}\Phi=8\frac{(\xi_1-\eta_1)^3}{\abs{\xi-\eta}^6}-6\frac{(\xi_1-\eta_1)^3}{\abs{\xi-\eta}^4}+8\frac{\eta_1^3}{\abs{\eta}^6}-6\frac{\eta_1^3}{\abs{\eta}^4},$$
so that 
$$\abs{\partial^2_{\eta_2}\Phi}\lesssim \frac{\abs{\eta_1}}{\abs{\eta}^4}.$$
Furthermore,
\begin{equation*}
\begin{aligned}
 \abs{\partial_{\eta_2}m(\xi,\eta)} &=\frac{1}{\abs{\eta}^4}\abs{-\xi_1\abs{\eta}^2-2\eta_2(\xi\cdot\eta^\perp)}\\
  &\leq \frac{1}{\abs{\eta}^4}\left(\frac{\abs{\eta_1}}{100}\abs{\eta}^2+2\abs{\eta}\abs{\eta_1}\abs{\eta} \right)\\
  &\leq \frac{3}{\abs{\eta}^2} \abs{\eta_1}.
\end{aligned}
\end{equation*}

Putting these together we obtain
$$\abs{\partial_{\eta_2} \left(\frac{m}{\partial_{\eta_2} \Phi}\right)}\lesssim \frac{\abs{\eta}^8}{\abs{\eta}^2\abs{\eta_1}^2}\frac{\abs{\eta_1}^2\abs{\eta}}{\abs{\eta}^2\abs{\eta}^4}\leq \abs{\eta}.$$
Due to the present localization to frequencies of comparable sizes and the homogeneities of the terms involved this yields \eqref{eq:Deta2phasemult}.
\end{proof}

\subsubsection{Region 2 -- \eqref{eq:G3}: $\abs{\xi}\ll\abs{\eta}\sim\abs{\xi-\eta}$}
Here $\abs{\xi}\leq \frac{\abs{\eta}}{100}$, so that $\abs{\xi-2\eta}\geq\abs{\eta}$ and we see that direct integration by parts in $\eta$ goes through, in analogy to Case 1 of Region 1. More precisely, from
\begin{equation}\label{eq:lossterm}
 \abs{\xi}^2\abs{m}\frac{\abs{\nabla_\xi\Phi}}{\abs{\nabla_\eta\Phi}}\lesssim_{CM}\abs{\eta}^2
\end{equation}
we get back the weight with two derivatives and lose one derivative in the energy estimates: once we notice that in this Region we have $\abs{\partial_\eta^2\Phi}\lesssim\abs{\eta}^{-4}\abs{\xi}$ it follows (see also the discussion for terms of type \eqref{it:t2} in Section \ref{sec:direct_approach})
\begin{equation*}
 m\,\nabla_\eta\cdot\left(\nabla_\eta\Phi\frac{\partial_{\xi_i}\Phi}{\abs{\nabla_\eta\Phi}^2}\right) \lesssim_{CM} \abs{\eta}^{-1}.
\end{equation*}

\subsubsection{Region 3 -- \eqref{eq:G4}: $\abs{\eta}\ll\abs{\xi}\sim\abs{\xi-\eta}$}
We have $\abs{\xi-2\eta}\geq\abs{\xi}\geq\abs{\eta}$, so we can proceed with the same integration by parts as in Region 2, given by \eqref{eq:G3}. One can check from the localization (into ``high'' and ``low'' frequencies) that the multipliers arising here are not of Coifman-Meyer type. We will instead use Theorem \ref{thm:multhm} to control them. By homogeneity of the constituent terms (and the frequencies being bounded by some $N>0$) one sees directly that any such multiplier $b(\xi,\eta)$ is bounded in $L^\infty_\xi\dot{H}^1_\eta$, and thus (see Lemma 5.1 in \cite{MR2775116}) one also has $\abs{b}_{\mathcal{M}^s}<\infty$ for $0\leq s<1$. More precisely, for $0<\rho\ll 1$ we let $s=1-\rho$ and have the bound $\abs{b}_{\mathcal{M}^{1-\rho}}\leq \rho^{-\frac{3}{2}}N^\rho$. By Theorem \ref{thm:multhm} the associated bilinear operator then is bounded\footnote{The point here is that a possible bound $L^2\times L^\infty\to L^2$ just barely fails, so in closing the bootstrap for Theorem \ref{thm:main} we will want to choose $\rho$ sufficiently small.} as a map $L^2\times L^{\frac{2}{\rho}}\to L^2$.

\begin{remark}\label{rem:lossterm}
 We point out that it is this region that requires us to work with two derivatives on the weight, rather than three. In all other regions three derivatives would give back three derivatives on the weight (and possibly a loss of derivative in the energy). However, here it can be seen that an additional factor of $\abs{\xi}$ on \eqref{eq:lossterm} would give a bound of $\abs{\xi}\abs{\eta}^2$ and thus only two derivatives on the weight, as well as a loss of a derivative in $L^\infty$. This is what forces us into the splitting into high and low frequencies.
\end{remark}

\subsection{Adding a Derivative}\label{sec:add_deriv}
Once we have established bounds for $\abs{D^2xf(t)}_{L^2}$ (Propositions \ref{prop:highfreq} and \ref{prop:lowfreq}) we can add another derivative: By going back to all terms involved we can ascertain that in this argument the bounds for two derivatives on a weight give rise to similar bounds for $\abs{D^3xf(t)}_{L^2}$, i.e.\ for three derivatives on a weight -- all one needs for this is the product rule of differentiation and the bound for the case of two derivatives. In direct analogy to Propositions \ref{prop:highfreq} and \ref{prop:lowfreq} one then obtains Proposition \ref{prop:added_derivative}.

Its proof consists of revisiting all terms and adding a derivative: By the product rule of differentiation this gives terms of the kind we had before with three derivatives on the weight instead of two (Regions 1 and 2), terms with two derivatives on the weight and one more derivative on the other profile (in the case of Region 3, see also Remark \ref{rem:lossterm} above) as well as terms without weights that can be dealt with by (higher order) energy estimates. All those can be estimated the way described here, the only change being a possible loss of an additional derivative.

\newpage
\appendix

\section{Proof of the Dispersive Estimate}\label{sec:disp_proof}
\begin{proof}[Proof of Proposition \ref{prop:disp_est}]
 We localize $g$ in frequency (using Littlewood-Paley theory), estimate the different pieces separately and sum in the end.
 
 \emph{1. Localization:} Denoting by $P_j$ the Littlewood-Paley projectors, by Young's convolution inequality we have
   \begin{equation*}
   \begin{aligned}
    \abs{e^{L_1 t}g}_{L^\infty}&=\abs{e^{L_1 t}\left(\sum_{j\in\Z} P_j g\right)}_{L^\infty}\leq \sum_{j\in\Z}\abs{e^{L_1 t}P_j g}_{L^\infty}\\
     &=\sum_{j\in\Z}\abs{\iFT{e^{-it\frac{\xi_1}{\abs{\xi}}}\varphi(2^{-j}\xi)\hat{g}(\xi))}}_{L^\infty}\\
     &=\sum_{j\in\Z}\abs{\iFT{e^{-it\frac{\xi_1}{\abs{\xi}}}\varphi(2^{-j}\xi)}\ast\iFT{\tilde{\varphi}(2^{-j}\xi)\hat{g}(\xi))}}_{L^\infty}\\
     &\leq \sum_{j\in\Z}\abs{\iFT{e^{-it\frac{\xi_1}{\abs{\xi}}}\varphi(2^{-j}\xi)}}_{L^\infty}\abs{\iFT{\tilde{\varphi}(2^{-j}\xi)\hat{g}(\xi))}}_{L^1}\\
     &=\sum_{j\in\Z}\abs{e^{L_1 t}\check{\varphi_j}}_{L^\infty}\abs{Q_j g}_{L^1},
   \end{aligned}
   \end{equation*}
 where $Q_j$ is a Littlewood-Paley projector associated to a slightly ``fattened'' bump function $\tilde{\varphi}$ (which equals $1$ on the support of $\varphi$).

 Thus it suffices to estimate $e^{L_1t}$ on the frequency localizers $\varphi$. To this end we notice that by a change of variables
 \begin{equation*}
  \begin{aligned}
   (e^{L_1 t}\check{\varphi_j})(x)&=\int_{\R^2}\varphi(2^{-j}\xi)e^{ix\cdot\xi -it\frac{\xi_1}{\abs{\xi}^2}}\, d\xi=2^{2j}\int_{\R^2}\varphi(\xi)e^{i(2^j x)\cdot\xi -i2^{-j}t\frac{\xi_1}{\abs{\xi}^2}}\, d\xi\\
    &=2^{2j}(e^{L_1 2^{-j}t}\check{\varphi})(2^j x).
  \end{aligned}
 \end{equation*}

 As will be shown further below we have
  \begin{equation}\label{eq:decay}
   \abs{e^{L_1 t}\check{\varphi}}_{L^\infty}\leq C t^{-1}
  \end{equation}
 with a constant depending only on the localizer $\varphi$. Hence our final estimate will be
  \begin{equation}\label{ineq:final_step}
  \begin{aligned}
  \abs{e^{L_1 t}g}_{L^\infty}&\leq \sum_{j\in\Z} 2^{2j}\abs{e^{L_1 2^{-j}t}\check{\varphi}}_{L^\infty}\abs{Q_j g}_{L^1}\leq \sum_{j\in\Z} 2^{2j}(2^{-j}t)^{-1}\abs{Q_j g}_{L^1}\\
  &=Ct^{-1}\abs{g}_{\dot{B}^{3}_{1,1}},
  \end{aligned}
 \end{equation}
 as claimed.

 \emph{2. Proof of \eqref{eq:decay}.} We recognize the estimate of 
  \begin{equation*}
   e^{L_1 t}\check{\varphi}(x)=\int_{\abs{\xi}\sim 1}\varphi(\xi)e^{ix\cdot\xi -it\frac{\xi_1}{\abs{\xi}^2}}\, d\xi = \int_{\abs{\xi}\sim 1}\varphi(\xi)e^{it\phi(\xi)}\, d\xi
  \end{equation*}
  in $L^\infty$ as a classical question in the theory of oscillatory integrals and resolve it using the method of stationary phase. Here the phase function $\phi$ is given as
  \begin{equation*}
  \phi(\xi):=\frac{x}{t}\cdot\xi-\frac{\xi_1}{\abs{\xi}^2}
  \end{equation*}
 and satisfies
 \begin{equation*}
  \nabla_\xi\phi=\frac{x}{t}-\abs{\xi}^{-4}\left(\xi_2^2-\xi_1^2,-2\xi_1\xi_2\right),
 \end{equation*}
 so that for every $(t,x)\in\R^+\times\R^2$ there are at most four stationary points. A direct computation then gives the determinant of the Hessian of $\phi$ as
 \begin{equation*}
  \det H_{\phi}=-\frac{4}{\abs{\xi}^{6}},
 \end{equation*}
 which is non-degenerate away from the origin and bounded above and below in the support of $\varphi$. Hence the standard stationary phase lemma concludes the proof.
\end{proof}

\section{Global Well-posedness (Without Decay) for Large Data}\label{sec:GWP}
In this section we prove that the $\beta$-plane equation is actually globally well-posed in $H^k$, $k>1$, for large data. However, a priori the $H^k$ norms of the solution could grow faster than double exponentially in time, which is completely different from the situation in the small data case. 

\begin{thm}\label{thm:GWP}
For each $k>1$ and $\omega_0\in H^k\cap H^{-1}$, the initial value problem 
\begin{equation}
  \begin{cases}
   &\partial_t\omega+u\cdot\nabla\omega=L_1\omega,\quad\\
   &\omega(0)=\omega_0\in H^k\cap \dot{H}^{-1}
  \end{cases}
\end{equation}
for the original $\beta$-plane equation has a unique solution $\omega\in C([0,\infty),H^k)$ for all time. Recall that here $u=\nabla^\perp(-\Delta)^{-1}\omega$ and $L_1\omega=\frac{R_1}{\abs{\nabla}}\omega$.
\end{thm}

\begin{remark}
Note that the $\dot{H}^{-1}$ control on $\omega$ is merely the $L^2$ conservation of the velocity field $u$. 
\end{remark}

\begin{remark}
The growth given in our proof is of triple-exponential order. However, it is possible to improve the proof to give an upper bound of double-exponential type -- this is done by carefully checking the constants in the Sobolev imbedding theorem and, instead of using $L^3$, estimating through $L^p$ for $p$ close to $2$. 
\end{remark}

The proof basically adapts the standard technique used to obtain the well-posed\-ness of the 2d Euler equation to our setting.

\begin{proof}
The main idea is to use the conservation of the vorticity (which has roughly the same regularity as $\nabla u$) to prove an a priori bound on the $H^k$ norms of the solution. This is done using a borderline Sobolev inequality of the following type:
$$\abs{\nabla u}_{L^\infty} \lesssim \abs{\omega}_{L^\infty}   \log\left(\frac{\abs{\omega}_{H^s}}{\abs{\omega}_{L^\infty}}\right),$$
for any $s>1,$ since $u=\nabla^\perp(-\Delta)^{-1}\omega$.

From this one estimates that $$\abs{\ip{u\cdot\nabla\omega,\omega}_{H^s}}\lesssim \abs{\nabla u}_{L^\infty} \abs{\omega}_{H^s}^2 $$ using, for example, the Kato-Ponce inequality and the fact that $u=\nabla^\perp (-\Delta)^{-1} \omega$.

This implies the energy estimate
$$\partial_t\abs{\omega}_{H^s}\lesssim \abs{\omega}_{L^\infty} \log\left(\frac{\abs{\omega}_{H^{s}}}{\abs{\omega}_{L^\infty}}\right)\abs{\omega}_{H^s}.$$

Now, \emph{if} we had a priori control on $\omega$ in $L^\infty$, say
$$\abs{\omega(t)}_{L^\infty}\leq A(t),$$
we would be able to control the $H^s$ estimates for all time using the Gr\"onwall inequality as follows:
\begin{equation*}
 \begin{aligned}
  \partial_t \abs{\omega}_{H^s} &\lesssim \abs{\omega}_{L^\infty} \log \left(\frac{\abs{\omega}_{H^{s}}}{\abs{\omega}_{L^\infty}}\right)\abs{\omega}_{H^s}\\
  &=\abs{\omega}_{L^\infty}\log(\abs{\omega}_{H^s}) \abs{\omega}_{H^s}-\abs{\omega}_{L^\infty}\log\left(\abs{\omega}_{L^\infty}\right)\abs{\omega}_{H^s}\\
  &\leq \left(A(t)\log\abs{\omega}_{H^s}+\frac{1}{e}\right)\abs{\omega}_{H^s} 
 \end{aligned}
\end{equation*}
since 
$$x \cdot \log\frac{1}{x}\leq \frac{1}{e}, \quad x\in [0,\infty).$$

Hence $$\partial_t \abs{\omega}_{H^s}\leq  \left(A(t)\log\abs{\omega}_{H^s}+\frac{1}{e}\right)\abs{\omega}_{H^s},$$ and so $$\partial_t \log \abs{\omega}_{H^s}\leq  \left(A(t)\log\abs{\omega}_{H^s}+\frac{1}{e}\right),$$
which shows that $\abs{\omega}_{H^s}$ is bounded in $L^\infty_{loc}(0,\infty)$. 

\emph{A priori global bound on $\omega$ in $L^\infty$.}
Our equation reads
$$\partial_t\omega+u\cdot\nabla\omega= \partial_1(-\Delta)^{-1}\omega.$$
Recall that $\ip{\partial_1(-\Delta)^{-1}\omega,\omega}_{L^2}=0$ and $\ip{u\cdot\nabla\omega,\omega}_{L^2}=0$, thus 
$$\abs{\omega(t)}_{L^2}=\abs{\omega_0}_{L^2}$$
and similarly
$$\abs{u(t)}_{L^2}=\abs{u_0}_{L^2},$$
just by noting that $\ip{\nabla^\perp (-\Delta)^{-1} \left(u\cdot\nabla\omega\right),u}_{L^2}=0.$ 

From $\partial_1 (-\Delta)^{-1}\omega=u_2$, the Sobolev embedding theorem and the boundedness of Riesz transforms on $L^3$ we deduce that
\begin{equation*}
 \begin{aligned}
  \abs{\partial_1(-\Delta)^{-1}\omega}_{L^\infty} &\lesssim \abs{\partial_1(-\Delta)^{-1}\omega}_{W^{1,3}} \lesssim \abs{u}_{L^3}+\abs{\omega}_{L^3} \leq \abs{u}_{H^1} +\abs{\omega}_{L^3}\\
  &\lesssim \abs{u_0}_{L^2}+\abs{\omega_0}_{L^2}+\abs{\omega}_{L^\infty}.
 \end{aligned}
\end{equation*}

We may also write the $\beta$-plane equation as:
$$\partial_t\omega + u\cdot\nabla\omega= F,$$ 
with
\begin{equation}\label{eq:Linftybnd}
 \abs{F}_{L^\infty} \leq C(\omega_0)+ \abs{\omega}_{L^\infty}.
\end{equation}

This gives, using the Gr\"onwall Lemma, a global bound on the $L^\infty$ norm of $\omega$.  Indeed, define the Lagrangian flow map associated to $u$ by 
\begin{equation*}
 \begin{cases}
  \frac{d}{dt} \Phi(t,x)&=u(t,\Phi(t,x)),\\ \Phi(0,x)&=x.
 \end{cases}
\end{equation*}
Then $$\frac{d}{dt} (\omega\circ\Phi)= F\circ\Phi.$$

We also have $\abs{\omega\circ\Phi}_{L^\infty}=\abs{\omega}_{L^\infty}$ and $\abs{F\circ\Phi}_{L^\infty}=\abs{F}_{L^\infty}$, which means that $$\abs{\omega}_{L^\infty} \leq  \abs{\omega_0}_{L^\infty} + \int_0^t \abs{F}_{L^\infty}(s)ds.$$
Combined with \eqref{eq:Linftybnd} we thus obtain an a-priori bound on $\omega$ in $L^\infty_{loc}(0,\infty)$, say 
$$\abs{\omega(t)}_{L^\infty}\leq A(t),$$
which is all that was left to show.
\end{proof}

\section{A Bound on \texorpdfstring{$\hat{f}$}{f} in \texorpdfstring{$L^\infty$}{L infty}}\label{sec:Linfty}
In this section we prove Corollary \ref{cor:Linfty}, i.e.\ we demonstrate that in the framework of Theorem \ref{thm:main} we have $\abs{\abs{\xi}^2\hat{f}(t,\xi)}_{L^\infty}\lesssim\epsilon^{\frac{1}{2}}$, provided that $\abs{\abs{\xi}^2\hat{\omega}_0(\xi)}_{L^\infty}\leq\epsilon$. This is a consequence of the estimates \eqref{eq:thmbounds} applied to the Duhamel formula \eqref{eq:DH} and the arguments involved are closely linked to those in Section \ref{sec:weights}.

\begin{proof}
Recall the Duhamel expression for $\hat{f}$,
\begin{equation*}
 \hat{f}(t,\xi)=\hat{f_0}(\xi)+\int_0^t\int_{\R^2}e^{is\Phi}m(\xi,\eta)\hat{f}(s,\xi-\eta)\hat{f}(s,\eta)d\eta ds.
\end{equation*}
If at least one of the profiles in this expression is localized to frequencies larger than some $N>0$ we have the bound
\begin{equation*}
 \abs{\int_{\R^2}e^{is\Phi}m(\xi,\eta)\widehat{P_{>N}f}(s,\xi-\eta)\hat{f}(s,\eta)d\eta}_{L^\infty}\lesssim N^{1-k}\abs{f}^2_{H^k}
\end{equation*}
for $k$ as in the proof of Theorem \ref{thm:main}.

Hence we may assume that all frequencies are bounded by $N>0$. As for the weighted estimate the idea is to use the oscillations of the phase $\Phi$ to obtain additional decay. To this end we recall that $\abs{\nabla_\eta\Phi}=\frac{\abs{\xi-2\eta}\abs{\xi}}{\abs{\xi-\eta}^2\abs{\eta}^2}$ and use divisions of Fourier space akin to those in section \ref{sec:splitting}.

Firstly, for small $\eta$ we directly obtain the bound
\begin{equation*}
 \abs{\abs{\xi}^2\int_{\abs{\eta}\leq\frac{\abs{\xi}}{100}} e^{is\Phi}m(\xi,\eta)\hat{f}(\xi-\eta)\hat{f}(\eta)d\eta}_{L^\infty}\lesssim N^2\abs{\abs{\xi}^2\hat{f}(\xi)}_{L^\infty}\abs{f}_{\dot{H}^{-1}}
\end{equation*}
by estimating one of the profiles in $L^2$ and the other in $L^\infty$ with derivatives.

Secondly, when $\abs{\eta}>\frac{\abs{\xi}}{100}$ we further split the Duhamel term into two pieces by writing
\begin{equation*}
\begin{aligned}
 \abs{\xi}^2\int_{\abs{\eta}>\frac{\abs{\xi}}{100}} e^{is\Phi}m\hat{f}\hat{f}d\eta=&\abs{\xi}^2\underbrace{\int_{\abs{\eta}>\frac{\abs{\xi}}{100},\,\abs{\xi-2\eta}>\rho}e^{is\Phi}m\hat{f}\hat{f}d\eta}_{=:A}\\
  &\quad+ \underbrace{\abs{\xi}^2\int_{\abs{\eta}>\frac{\abs{\xi}}{100},\,\abs{\xi-2\eta}\leq\rho} e^{is\Phi}m\hat{f}\hat{f}d\eta}_{\leq\rho N \abs{\abs{\xi}^2\hat{f}(\xi)}_{L^\infty}\abs{f}_{\dot{H}^{-1}}},
\end{aligned}
\end{equation*}
where $\rho>0$ remains to be chosen later.

The first term of these is then amenable to an integration by parts in $\eta$, and the only thing left to do is to ascertain that the singularities that arise are controlled. We obtain
\begin{equation*}
\begin{aligned}
 &A=\int_{\abs{\eta}>\frac{\abs{\xi}}{100},\,\abs{\xi-2\eta}>\rho}\frac{1}{is}\nabla_\eta\Phi\cdot\nabla_\eta e^{is\Phi}\frac{m}{\abs{\nabla_\eta\Phi}^2}\hat{f}(\xi-\eta)\hat{f}(\eta)d\eta\\
&\quad \sim\frac{1}{is}\int_{\abs{\eta}>\frac{\abs{\xi}}{100},\,\abs{\xi-2\eta}>\rho}e^{is\Phi}\nabla_\eta\cdot\left(\nabla_\eta\Phi\frac{m}{\abs{\nabla_\eta\Phi}^2}\right)\hat{f}(\xi-\eta)\hat{f}(\eta)d\eta \quad=:I\\
&\qquad +\frac{1}{is}\int_{\abs{\eta}>\frac{\abs{\xi}}{100},\,\abs{\xi-2\eta}>\rho}e^{is\Phi}\nabla_\eta\Phi\frac{m}{\abs{\nabla_\eta\Phi}^2}\nabla_\eta\hat{f}(\xi-\eta)\hat{f}(\eta)d\eta \quad=:I\!I\\
&\qquad +\frac{1}{is}\int_{\abs{\eta}>\frac{\abs{\xi}}{100},\,\abs{\xi-2\eta}>\rho}e^{is\Phi}\nabla_\eta\Phi\frac{m}{\abs{\nabla_\eta\Phi}^2}\hat{f}(\xi-\eta)\nabla_\eta\hat{f}(\eta)d\eta \quad=:I\!I\!I.
\end{aligned}
\end{equation*}

Both terms $I\!I$ and $I\!I\!I$ can be estimated in $L^\infty$ by
\begin{equation*}
 \abs{\abs{\xi}^2 I\!I}_{L^\infty},\abs{\abs{\xi}^2 I\!I\!I}_{L^\infty}\leq \frac{N}{\rho s}\abs{\abs{\eta}^{2}\hat{f}}_{L^2} \abs{\abs{\eta}^{2}\nabla_\eta\hat{f}}_{L^2}
\end{equation*}
since as noted above $\abs{\xi}^2\abs{\nabla_\eta\Phi\frac{m}{\abs{\nabla_\eta\Phi}^2}}=\abs{\xi}^2\frac{\abs{m}}{\abs{\nabla_\eta\Phi}}\lesssim_{CM} N \rho^{-1}\abs{\xi-\eta}^2\abs{\eta}^2$ in this domain of integration.

To control $I$ we compute (akin to the remarks in Section \ref{sec:direct_approach})
\begin{equation*}
 \begin{aligned}
  m\,\nabla_\eta\cdot\left(\frac{\nabla_\eta\Phi}{\abs{\nabla_\eta\Phi}^2}\right)&\lesssim\frac{\abs{\xi-2\eta}\abs{\eta}}{\abs{\eta}^2}\frac{\abs{\xi-\eta}^4\abs{\eta}^4}{\abs{\xi-2\eta}^2\abs{\xi}^2}(\abs{\xi-\eta}^{-3}+\abs{\eta}^{-3})\\
  &\leq\frac{1}{\abs{\xi-2\eta}\abs{\xi}^2}(\abs{\xi-\eta}\abs{\eta}^3+\abs{\xi-\eta}^4)\\
 \end{aligned}
\end{equation*}
and
\begin{equation*}
  \frac{\abs{\nabla_\eta m}}{\abs{\nabla_\eta\Phi}}\sim \frac{\abs{\xi}\abs{\xi-\eta}^2\abs{\eta}^2}{\abs{\eta}^2\abs{\xi-2\eta}\abs{\xi}}\lesssim\frac{\abs{\xi-\eta}^2}{\abs{\xi-2\eta}}.
\end{equation*}

We further subdivide the region of integration in $I$: On the one hand this gives
\begin{equation*}
 \frac{1}{s}\abs{\abs{\xi}^2\int_{\abs{\xi-2\eta}\geq 1}e^{is\Phi}\nabla_\eta\cdot\left(\nabla_\eta\Phi\frac{m}{\abs{\nabla_\eta\Phi}^2}\right)\hat{f}(\xi-\eta)\hat{f}(\eta)d\eta}_{L^\infty}\lesssim s^{-1}\abs{f}_{H^k}^2.
\end{equation*}
On the other hand we note that $\int_{1>\abs{\xi-2\eta}>\rho}\frac{1}{\abs{\xi-2\eta}^2}d\eta\lesssim \abs{\log(\rho)}$, so estimating $\abs{\xi-2\eta}^{-1}$ in $L^2$, one of the profiles in $L^2$ and the other one in $L^\infty$ in the region where $1>\abs{\xi-2\eta}>\rho$ leaves us with the bound
\begin{equation*}
 \abs{\abs{\xi}^2 I}_{L^\infty}\lesssim\frac{1}{s} N^2\abs{\log(\rho)}\abs{f}_{H^2}\abs{\abs{\xi}^2\hat{f}(\xi)}_{L^\infty}+s^{-1}\abs{f}_{H^k}^2.
\end{equation*}

Collecting the above estimates shows that
\begin{equation*}
 \begin{aligned}
  \partial_t\abs{\abs{\xi}^2\hat{f}(\xi)}_{L^\infty}&\lesssim (\rho N+N^2) \abs{f}_{\dot{H}^{-1}}\abs{\abs{\xi}^2\hat{f}(\xi)}_{L^\infty}+\rho^{-1}s^{-1}N^3\abs{f}_{L^2}\abs{\abs{\xi}^2\partial_\xi\hat{f}(\xi)}_{L^2}\\
  &\quad+s^{-1}\abs{f}_{H^k}^2+s^{-1}N^2\abs{\log(\rho)}\abs{f}_{H^2}\abs{\abs{\xi}^2\hat{f}(\xi)}_{L^\infty}+N^{1-k}\abs{f}^2_{H^k}.
 \end{aligned}
\end{equation*}

Inserting now the bounds for $f$ and revisiting the arguments in the bootstrapping scheme in Section \ref{sec:bootstrap} yields the claim, once we let $\rho=s^{-1}$ (recall also that $\abs{f}_{\dot{H}^{-1}}=\abs{u}_{L^2}$).
\end{proof}

\section{Some Computations}\label{apdx:comp}
The following claims are the result of straightforward computations.
\paragraph{About the Null Forms}\label{apdx:nf}
Recall that $m(\xi,\eta)=\frac{\xi\cdot\eta^\perp}{\abs{\eta}^2}$ and $\bar{m}(\xi,\eta):=m(\xi,\xi-\eta)$.
\begin{lemma}\label{lem:nf_deriv}
 We have
\begin{equation}
 \begin{aligned}
  \nabla_\xi m(\xi,\eta)&=\frac{\eta^\perp}{\abs{\eta}^2}, \\
  \nabla_\xi\bar{m}(\xi,\eta)&=\frac{\eta^\perp}{\abs{\xi-\eta}^2}-2\frac{\xi-\eta}{\abs{\xi-\eta}^2}\bar{m}(\xi,\eta),\\
  \nabla_\eta m(\xi,\eta)&=-\frac{\xi^\perp}{\abs{\eta}^2}-2\frac{\eta}{\abs{\eta}^2} m(\xi,\eta),\\
  \nabla_\eta\bar{m}(\xi,\eta)&=-\frac{\xi^\perp}{\abs{\xi-\eta}^2}-2\frac{\xi-\eta}{\abs{\xi-\eta}}\bar{m}(\xi,\eta).
 \end{aligned}
\end{equation}
\end{lemma}

\paragraph{Some Phase Derivatives}\label{apdx:deriv}
Direct computation gives the following derivatives.

\begin{equation}
 \begin{aligned}
  \partial_{\eta_1}\Phi&=-\frac{\eta_2^2-\eta_1^2}{\abs{\eta}^4}+\frac{(\xi_2-\eta_2)^2-(\xi_1-\eta_1)^2}{\abs{\xi-\eta}^4}\\
  \partial_{\eta_2}\Phi&=-2\frac{\eta_1\eta_2}{\abs{\eta}^4}+2\frac{(\xi_1-\eta_1)(\xi_2-\eta_2)}{\abs{\xi-\eta}^4}
 \end{aligned}
\end{equation}

\begin{equation}
 \abs{\nabla_\eta\Phi}=\frac{\abs{\xi-2\eta}\abs{\xi}}{\abs{\xi-\eta}^2\abs{\eta}^2}, \quad\abs{\nabla_\xi\Phi}=\frac{\abs{\eta-2\xi}\abs{\eta}}{\abs{\xi-\eta}^2\abs{\xi}^2}
\end{equation}

\begin{equation}
 \begin{aligned}
  \partial_{\eta_1}^2\Phi=-\partial_{\eta_2}^2\Phi=-8\frac{(\xi_1-\eta_1)^3}{\abs{\xi-\eta}^6}+6\frac{(\xi_1-\eta_1)}{\abs{\xi-\eta}^4}-8\frac{\eta_1^3}{\abs{\eta}^6}+6\frac{\eta_1}{\abs{\eta}^4}
 \end{aligned}
\end{equation}

\begin{equation}\label{eq:mixedpartials}
 \begin{aligned}
  \partial_{\xi_1}\partial_{\eta_1}\Phi&=-\partial_{\xi_2}\partial_{\eta_2}\Phi=\frac{2(\xi_1-\eta_1)\left[(\xi_1-\eta_1)^2-3(\xi_2-\eta_2)^2\right]}{\abs{\xi-\eta}^6}\\
  \partial_{\xi_2}\partial_{\eta_1}\Phi&=\partial_{\xi_1}\partial_{\eta_2}\Phi=\frac{2(\xi_2-\eta_2)\left[-3(\xi_1-\eta_1)^2+(\xi_2-\eta_2)^2\right]}{\abs{\xi-\eta}^6}\\
 \end{aligned}
\end{equation}

\begin{acknowledgements}
 The authors would like to thank Pierre Germain and Fabio Pusateri for many helpful discussions. We are also grateful to the anonymous referee for his/her careful reading of the paper and the insightful comments.
 
 T.\ M.\ Elgindi acknowledges funding from NSF grant DMS-1402357.
\end{acknowledgements}

\bibliographystyle{plain}
\bibliography{bplanerefs.bib}

\end{document}